\documentclass[11pt]{article}
\usepackage{amsmath}
\usepackage{amsthm}
\usepackage{amssymb}
\usepackage{fullpage}
\usepackage{color}
\usepackage{cancel}
\usepackage{soul}
\usepackage{authblk}
\usepackage{marvosym}

\newcommand{\comment}[1]{}

\newcommand{\be}{\begin{equation}}
\newcommand{\bel}[1]{\begin{equation}\label{#1}}
\newcommand{\qe}{\end{equation}}
\newcommand{\ee}{\end{equation}}
\newcommand{\eeq}{\end{equation}}
\newcommand{\ba}{\begin{eqnarray}}
\newcommand{\ea}{\end{eqnarray}}

\numberwithin{equation}{section}

\theoremstyle{plain}
\newtheorem{theorem}{Theorem}[section]

\theoremstyle{definition}
\newtheorem{definition}{Definition}[section]   
 
\theoremstyle{plain}

\theoremstyle{plain}
\newtheorem{proposition}{Proposition}[section]         %%%% occasional
        
\theoremstyle{plain}

\theoremstyle{definition}
\newtheorem{remark}{Remark}[section]
\theoremstyle{definition}
\newtheorem{example}{Example}[section]

\newcommand{\R}{{\mathbb R}}

\newcommand{\N}{{\mathbb N}}
\newcommand{\Q}{{\mathbb Q}}

\newcommand{\Mm}{{\mathcal M}}    %moduli space

\newcommand{\Pp}{{\mathcal P}}

\newcommand{\Ss}{{\mathcal S}}

\newcommand{\Om}{{\Omega}}

\newcommand{\eps}{{\varepsilon}}

\newcommand{\p}{{\partial}}

\newcommand{\lra}{\longrightarrow}

\newcommand{\sign}{\mbox{sign}}

\newcommand{\g}{{\mathfrak g}} 
\newcommand{\T}{{\mathbf T}}

\definecolor{brown}{RGB}{150,100,0}

\renewcommand{\sign}{{\rm sign}\,}

\begin{document}
%
%\begin{frontmatter}

\title{Parametrized measure models}
%\runtitle{Parametrized measure models}

%\begin{aug}
%
%%% inicialai - be tarpu
%
\author[1]{Nihat Ay}
\author[1]{J\"urgen Jost}
\author[2]{H\^ong V\^an L\^e}
\author[3]{Lorenz Schwachh\"ofer}
\affil[1]{Max-Planck-Institute for Mathematics in the Sciences, Leipzig, Germany, \newline \Letter{\footnotesize nay@mis.mpg.de, jjost@mis.mpg.de}}
\affil[2]{Academy of Sciences of the Czech Republic, Prague, \Letter {\footnotesize hvle@math.cas.cz}}
\affil[3]{TU Dortmund University, Dortmund, Germany, \Letter {\footnotesize lschwach@math.tu-dortmund.de}}
\maketitle 

% ABSTRACT
%
\begin{abstract}
We develop a new and general notion of parametric measure models and
statistical models on an arbitrary sample space $\Omega$ which does not
assume that all measures of the model have the same null sets. This is
given by a differentiable map from the parameter manifold $M$ into the
set of finite measures or probability measures on $\Omega$, respectively,
which is differentiable when regarded as a map into the Banach space of
all signed measures on $\Omega$. Furthermore, we also give a rigorous
definition of roots of measures and give a natural characterization of
the Fisher metric and the Amari--Chentsov tensor as the pullback of
tensors defined on the space of roots of measures. We show that many
features such as the preservation of this tensor under sufficient
statistics and the monotonicity formula hold even in this very general set-up.
\end{abstract}

% KEYWORDS
%
\noindent {\em MSC2010: 53C99, 62B05}

\noindent {\em Keywords: Fisher quadratic form, Amari-Chentsov tensor, sufficient statistic, monotonicity}

%
%\tableofcontents[level=?]%% paliekamas kai>=50 psl., level pagal MS
%%%%%%%%%%%%%%%%%%%%%%%%%%%%%%%%%%%%%%%%%%%%%%%%%%%%%%%%%%%%%%%%%%%%

%s1 #&#
\section{Introduction}\label{sec1}

Information geometry is concerned with the use of differential
geometric methods in probability theory. An important object of
investigation are families of probability measures or, more generally,
of finite measures on a given sample space $\Om$ which depend
differentiably on a finite number of parameters. Associated to such a
family there are two symmetric tensors on the parameter space $M$. The
first is a quadratic form (i.e., a Riemannian metric), called the \emph{Fisher metric} $\g^F$, and the second is a $3$-tensor, called the
\emph{Amari--Chentsov tensor} $\T^{\mathrm{AC}}$. The Fisher metric was first
suggested by Rao \cite{Rao1945}, followed by Jeffreys \cite
{Jeffreys1946}, Efron \cite{Efron1975} and then systematically
developed by Chentsov and Morozova \cite{Chentsov1965,Chentsov1978} and \cite{MC1991}; the Amari--Chentsov tensor and its
significance was discovered by Amari \cite{Amari1980,Amari1982} and Chentsov \cite{Chentsov1982}.

These tensors are of interest from the differential geometric point of
view as they do not depend on the particular choice of parametrization
of the family, but they are also natural objects from the point of view
of statistics, as they are unchanged under sufficient statistics and
are in fact characterized by this property; this was shown in the case
of finite sample spaces by Chentsov in \cite{Chentsov1978} and more
recently for general sample spaces in \cite{AJLS}. In fact, Chentsov
not only showed the invariance of these tensors under sufficient
statistics, but also under what he called \emph{congruent embeddings} of
probability measures. These are Markov kernels between finite sample
spaces which are right inverses of a statistic. We use this property to
give a definition of congruent embeddings between arbitrary sample
spaces (cf. Definition~\ref{def:cong-emb}). As it turns out, every
Markov kernel induces a congruent embedding in this sense, but there
are congruent embeddings which are \emph{not} induced by Markov
kernels, cf. Theorem~\ref{Markov-congruent}.

The main conceptual difficulty in the investigation of families of
probability measures is the lack of a canonical manifold structure on
the spaces $\Mm(\Om)$ and $\Pp(\Om)$ of finite measures and
probability measures on $\Om$. If $\Om$ is finite, then a measure is
given by finitely many nonnegative parameters, allowing to identify
$\Mm(\Om)$ with the closure of the positive orthant in $\R^{|\Om|}$
and $\Pp(\Om)$ with the intersection of this closure with an affine
hyperplane in $\R^{|\Om|}$, so that both are (finite dimensional)
manifolds with corners. If one does not assume that all elements of
$\Om$ have positive mass for all measures in the family, that is,
allowing the model to contain elements of the boundary of $\Mm(\Om)$
or $\Pp(\Om)$, then technical difficulties arise for example, when
describing the Fisher metric and the Amari--Chentsov tensor. If $\Om$
is infinite, then there is a priori not even a differentiable structure
on $\Mm(\Om)$ and $\Pp(\Om)$.

Attempts have been made to provide $\Pp(\Om)$ and $\Mm(\Om)$ with a
Banach manifold structure. For instance, Pistone and Sempi \cite
{PS1995} equipped these spaces with a topology, the so-called
$e$-\emph{topology}. With this, $\Pp(\Om)$ and $\Mm(\Om)$ become Banach
manifolds and have many remarkable features, see, for example, \cite
{CP2007,P2013}. On the other hand, the $e$-topology is very
strong in the sense that many families of measures on $\Om$ fail to be
continuous w.r.t. the $e$-topology, so it cannot be applied as widely
as one would wish.

Another approach was recently pursued by Bauer, Bruveris and Michor
\cite{BBM} under the assumption that $\Om$ is a manifold. In this
case, the space of smooth densities also carries a natural topology,
and they were able to show that the invariance under diffeomorphisms
already suffices to characterize the Fisher metric of a family of such
densities.

In \cite{AJLS}, the authors of the present article proposed to define
parametrized measure models as a family $(\mathbf{p}(\xi))_{\xi\in M}$
of finite measures on $\Om$, labelled by elements $\xi$ of a finite
dimensional manifold $M$, such that for a measurable set $A \subset\Om$
%
%e1.1 #&#
\begin{equation}
\label{p-infinite2} \mathbf{p}(\xi) (A) =
\int_A d\mathbf{p}(\xi) =
\int_A p({\omega}; \xi) \, d\mu({\omega})
\end{equation}
for some reference measure $\mu$ and a positive function $p$ on $\Om
\times M$ which is differentiable in $\xi\in M$. This closely follows
the notion of Amari \cite{Amari1980}. That is, for fixed $\xi\in M$,
the function $p(\cdot;\xi)$ on $\Om$ is the Radon--Nikodym derivative
of the measure $\mathbf{p}(\xi)$ w.r.t. $\mu$, whence by a slight abuse
of notation (e.g., \cite{Bauer}, Definition~17.2) we abbreviate (\ref
{p-infinite2}) as
%
%e1.2 #&#
\begin{equation}
\label{p-infinite} \mathbf{p}(\xi) = p(\cdot; \xi) \mu.
\end{equation}

While this notion embraces many interesting families of measures, it is
still restricted as it requires the existence of a reference measure
$\mu$ dominating all measures $\mathbf{p}(\xi)$, and on the other hand,
the positivity of the density function implies that all measures
$\mathbf{p}(\xi)$ on $\Om$ are equivalent, that is, have the same null sets.
While the existence of a measure $\mu$ dominating all measures
$\mathbf{p}(\xi)$ is satisfied for example, if $M$ is a finite dimensional
manifold, the condition that all measures $\mathbf{p}(\xi)$ have the same
null sets is a more severe restriction of the admissible families.

It is the aim of the present article to provide a yet more general
definition of parametrized measure models which embraces all of the
aforementioned definitions, but is more general and more natural than
these at the same time. Namely, in this article we define \emph{parametrized measure models} and \emph{statistical models},
respectively, as families $(\mathbf{p}(\xi))_{\xi\in M}$ which are given
by a map $\mathbf{p}$ from $M$ to $\Mm(\Om)$ and $\Pp(\Om)$,
respectively, which is differentiable when regarded as a map between
the (finite or infinite dimensional) manifold $M$ and the Banach space
$\Ss(\Om)$ of finite signed measures on $\Om$, since evidently $\Pp
(\Om)$ and $\Mm(\Om)$ are subsets of $\Ss(\Om)$. That is, the
geometric structure on $\Mm(\Om)$ and $\Pp(\Om)$ is given by the
inclusions $\Pp(\Om) \hookrightarrow\Mm(\Om) \hookrightarrow\Ss
(\Om)$.

For the models defined in \cite{AJLS}, the function $p: \Om\times M
\to\R$ in (\ref{p-infinite}) is differentiable into the $\xi
$-direction, and such a $p$ is called a \emph{regular density function}.
Even if a parametrized measure model in the sense of the present paper
has a dominating measure $\mu$ and hence is given by (\ref
{p-infinite}), the density function $p$ is not necessarily regular, cf.
Remark~\ref{rem:paramMeas} and Example~\ref{ex4.2}(2) below, and
$p$ is not required to be positive $\mu$-a.e., making this notion more
general than that in \cite{AJLS}. We shall show that most of the
statements shown in \cite{AJLS} for parametrized measure models or
statistical models with a positive regular density function also hold
in this more general setup.

The Fisher metric $\g^F$ and the Amari--Chentsov tensor $\T^{\mathrm{AC}}$
associated to a parametrized measure model are the two symmetric
tensors given by
\begin{eqnarray*}
\g^F(V, W) & := &
\int_\Om\p_V \log p({\omega};\xi) \,
\p_W \log p({\omega};\xi) \, d\mathbf{p}(\xi),
\\
\mathbf{T}^{\mathrm{AC}}(V, W, U) & := &
\int_\Om\p_V \log p({\omega};\xi)
\, \p_W \log p({\omega};\xi)\, \p_U \log p({\omega};\xi) \,
d\mathbf{p}(\xi).
\end{eqnarray*}

The crucial observation is that even though the function $\log
p({\omega};\xi)$ is not defined everywhere if we drop the assumption
that the density function $p$ is positive, the partial derivatives $\p
_V \log p({\omega};\xi)$ still may be given sense for an arbitrary
parametrized measure model. Thus, the notion of $k$-integrability from
\cite{AJLS} requiring that $\p_V \log p({\omega};\xi) \in L^k(\Om,
p(\xi))$ for all $V \in T_\xi M$ generalizes to parametrized measure models.

We also introduce the Banach space $\Ss^r(\Om)$ of $r$th powers of
measures on $\Om$ for $r \in(0,1]$, which has been discussed in \cite{Neveu1965}, Ex. IV.1.4, for general $\Om$ and generalizes the concept
of half densities on a manifold $\Om$ in \cite{MR}, Section 6.9.1. The
elements of $\Ss^r(\Om)$ can be raised to the $1/r$th power to
become finite signed measures, and for each measure $\mu\in\Mm(\Om)
\subset\Ss(\Om)$ there is a well defined power $\mu^r \in\Ss
^r(\Om)$. Thus, for a parametrized measure model $\mathbf{p}: M \to\Mm
(\Om)$ the $r$th power defines a map $\mathbf{p}^r: M \to\Ss^r(\Om)$,
and if the model is $k$-integrable for $k = 1/r \geq1$, then
$\mathbf{p}^r$ is differentiable, and for $k = 2$ or $k=3$, $\g^F$ and
$\T^{\mathrm{AC}}$ are pull-backs of canonical tensors on $\Ss^{1/2}(\Om)$
under $\mathbf{p}^{1/2}$ and $\Ss^{1/3}(\Om)$ under $\mathbf{p}^{1/3}$,
respectively.

We also discuss the behavior of the Fisher metric under statistics,
i.e., under measurable maps $\kappa: \Om\to\Om'$ or, more general,
under Markov kernels $K: \Om\to\Pp(\Om')$. These transitions can be
interpreted as data processing in statistical decision theory, which
can be deterministic (given by a measurable map, i.e., a statistic) or
randomized (i.e., given by a Markov kernel). The earliest occurrence of
this point of view appears to be \cite{Chentsov1982}.

Given a parametrized measure model $\mathbf{p}: M \to\Mm(\Om)$, it
induces a map $\mathbf{p}'(\xi) := \kappa_* \mathbf{p}(\xi)$ or
$\mathbf{p}'(\xi) := K_* \mathbf{p}(\xi)$, respectively. We show that this
process preserves $k$-integrability, i.e., if $\mathbf{p}$ is
$k$-integrable, then so is $\mathbf{p}'$ (cf. Theorem~\ref
{thm:induced-kintegrable}). Moreover, in Theorem~\ref
{thm:monotonicity} we show in this general setup the estimate
%
%e1.3 #&#
\begin{equation}
\label{eq:mono-I} \bigl\|\p_V \log p(\cdot;\xi)\bigr\|
_k \geq\bigl\|\p_V \log p'(\cdot;
\xi)\bigr\|_k,\qquad \mbox{whence } \g^F(V, V) \geq{
\g'}^F(V, V),
\end{equation}
where the second estimate is called the \emph{monotonicity formula} and
follows form the first for \mbox{$k = 2$}. The difference $\|\p_V \log
p(\cdot;\xi)\|_k^k - \|\p_V \log p'(\cdot;\xi)\|_k^k \geq0$ is
called the $k$\emph{th order information loss under} $\kappa$ (\emph{or} $K$)
\emph{in direction} $V$. If the information loss in any direction vanishes, then we call the statistic {\em sufficient} for the model.

There is a remarkable difference between parametrized measure models
with \emph{positive} regular density functions, that is, those
considered in \cite{AJLS}, and the more general notion establishes in
this paper. Namely, in case of a positive regular density function the
vanishing of the information loss for a statistic $\kappa: \Om\to\Om
'$ implies that the statistic admits a Fisher-Neyman factorization, cf. Proposition~\ref{prop:infoloss}. Remarkably, this is no longer true in our
setting. That is, if we admit parametrized measure models with
inequivalent measures, then there are statistics which have vanishing
information loss, but do not admit a Fisher-Neyman factorization, cf. Example~\ref{ex:suff}.

This paper is structured as follows. In Section~\ref{sec:measures}, we
give the formal definition of the spaces of $r$th powers of measures.
In Section~\ref{sec:congruent}, we provide a precise definition of
congruent embeddings for arbitrary sample spaces $\Om$ and discuss
their relations with Markov kernels and the existence of transverse
measures. In the following Section~\ref{sec:k-int}, we establish the
notion of $k$-integrability, which is applied in the final
Section~\ref{sec:suffstat} to the discussion of sufficient statistics and the proof
of the monotonicity formula.

%%%%%%%%%%%%%%%%%%%%%%%%%%%%%%%%%%%%
%s2 #&#
\section{The spaces of measures and their powers} \label{sec:measures}
%%%%%%%%%%%%%%%%%%%%%%%%%%%%%%%%%%%%

%%%%%%%%%%%%%%%%%%%%%%%%%%%%%%%%%%%%
%s2.1 #&#
\subsection{The space of (signed) finite measures}\label{sec2.1}
%%%%%%%%%%%%%%%%%%%%%%%%%%%%%%%%%%%%

Let $(\Om, \Sigma)$ be a measurable space, that is an arbitrary set
$\Om$ together with a sigma algebra $\Sigma$ of subsets of $\Om$.
Regarding the sigma algebra $\Sigma$ on $\Om$ as fixed, we let
\begin{eqnarray*}
\Pp(\Om) & := & \{ \mu : \mu \mbox{ a probability measure on } \Om\},
\\
\Mm(\Om) & := & \{ \mu : \mu \mbox{ a finite measure on }\Om \},
\\
\Ss(\Om) & := & \{ \mu : \mu \mbox{ a signed finite measure on } \Om\},
\\
\Ss_0(\Om) & := & \biggl\{ \mu\in\Ss(\Om) :
\int_\Om d\mu= 0\biggr\}.
\end{eqnarray*}

Clearly, $\Pp(\Om) \subset\Mm(\Om) \subset\Ss(\Om)$, and $\Ss
_0(\Om), \Ss(\Om)$ are real vector spaces. In fact, both $\Ss_0(\Om
)$ and $\Ss(\Om)$ are Banach spaces whose norm is given by the total
variation of a signed measure, defined as
\[
{\|\mu\|}_{\mathrm{TV}} := \sup\sum_{i = 1}^n
\bigl\|\mu(A_i)\bigr\|,
\]
where the supremum is taken over all finite partitions $\Omega= A_1\, 
\dot\cup\, \cdots\, \dot\cup\,  A_n$ with disjoint sets \mbox{$A_i \in\Sigma$}.
Here, the symbol $\dot\cup$ stands for the disjoint union of sets.

For a measurable function $\phi: \Om\to[-\infty, \infty]$, we
define $\phi_+ := \max(\phi, 0)$ and $\phi_- := \max(-\phi, 0)$,
so that $\phi_\pm\geq0$ are measurable with disjoint support, and
%
%e2.1 #&#
\begin{equation}
\label{eq:phi+-} \phi= \phi_+ - \phi_-,\qquad |\phi| = \phi_+ + \phi_-.
\end{equation}
By the \emph{Jordan decomposition theorem}, each measure $\mu\in\Ss
(\Om)$ can be decomposed uniquely as
%
%e2.2 #&#
\begin{equation}
\label{Jordan-dec} \mu= \mu_+ - \mu_- \qquad \mbox{with }\mu_\pm\in\Mm(
\Om), \mu_+ \perp\mu_-.
\end{equation}
That is, there is a decomposition $\Om= P \dot\cup N$ with $\mu_+(N)
= \mu_-(P) = 0$. Thus, if we define
\[
|\mu| := \mu_+ + \mu_- \in\Mm(\Om),
\]
then (\ref{Jordan-dec}) implies
%
%e2.3 #&#
\begin{equation}
\label{Jordan-dec2} \bigl\|\mu(A)\bigr\|\leq|\mu|(A) \qquad \mbox{for all
}\mu\in \Ss(\Om)\mbox{ and }A \in\Sigma,
\end{equation}
so that
\[
\|\mu\|_{\mathrm{TV}} = \bigl\|\|\mu\|\bigr\|
_{\mathrm{TV}} = |\mu|(\Om).
\]
In particular,
\[
\Pp(\Om) = \bigl\{ \mu\in\Mm(\Om) : \|\mu\|_{\mathrm{TV}}= 1\bigr
\}.
\]
Moreover, fixing a measure $\mu_0 \in\Mm(\Om)$, we let
%
%e2.4 #&#
\begin{eqnarray}
\label{def-S(Om,mu)} \begin{aligned} \Pp(\Om, \mu_0) &:= \bigl\{ \mu \in
\Pp(\Om) : \mu \mbox{ is dominated by } \mu_0 \bigr\},
\\
\Mm(\Om, \mu_0) &:= \bigl\{ \mu \in\Mm(\Om) : \mu \mbox{ is
dominated by } \mu_0 \bigr\},
\\
\Pp_+(\Om, \mu_0) &:= \bigl\{ \mu \in\Pp(\Om, \mu_0) :
\mu \mbox{ is equivalent to } \mu_0 \bigr\},
\\
\Mm_+(\Om, \mu_0) &:= \bigl\{ \mu \in\Mm(\Om, \mu_0) :
\mu \mbox{ is equivalent to } \mu_0 \bigr\},
\\
\Ss(\Om, \mu_0) &:= \bigl\{ \mu \in\Ss(\Om) : \mu \mbox{ is
dominated by } \mu_0 \bigr\},
\\
\Ss_0(\Om, \mu_0) &:= \Ss(\Om, \mu_0)
\cap\Ss_0(\Om), \end{aligned}
\end{eqnarray}
where we say that $\mu_0$ dominates $\mu$ if every $\mu_0$-null set
is also a $|\mu|$-null set and where we call two measures equivalent
if they dominate each other and hence have the same null sets. The
spaces in (\ref{def-S(Om,mu)}) do not change when replacing $\mu_0$
by an equivalent measure.

We may canonically identify $\Ss(\Om, \mu_0)$ with $L^1(\Om, \mu
_0)$ by the correspondence
\[
\imath_{\mathrm{can}}: L^1(\Om, \mu_0) \longrightarrow
\Ss(\Om, \mu_0),\qquad \phi\longmapsto\phi \mu_0.
\]
By the Radon--Nikodym theorem, this is an isomorphism whose inverse is
given by the Radon--Nikodym derivative $\mu\mapsto\frac{d\mu}{d\mu
_0}$. With this, $\Mm(\Om, \mu_0) = \{ \phi\mu_0  :  \phi\geq
0\}$ and $\Mm_+(\Om, \mu_0) = \{ \phi\mu_0  :  \phi> 0\}$ and
the corresponding descriptions apply to $\Pp(\Om, \mu_0)$ and $\Pp
_+(\Om, \mu_0)$, respectively.
Observe that $\imath_{\mathrm{can}}$ is an isomorphism of Banach spaces, since evidently
\[
\|\phi\|_{L^1(\Om, \mu_0)} =
\int_\Om|\phi| \, d\mu_0 = \|\phi
\mu_0\|_{\mathrm{TV}}.
\]

%%%%%%%%%%%%%%%%%%%%%%%%%%%%%%%%%%%%
%s2.2 #&#
\subsection{Differential maps between Banach manifolds and tangent
double cone fibrations} \label{sec:Banach}
%%%%%%%%%%%%%%%%%%%%%%%%%%%%%%%%%%%%

In this section, we shall recall some basic notions of maps between
Banach manifolds. For simplicity, we shall restrict ourselves to maps
between open subsets of Banach spaces, even though this notion can be
generalized to general Banach manifolds, see, for example, \cite{Lang2002}.

Let $V$ and $W$ be Banach spaces and $U \subset V$ an open subset. A
map $\phi: U \to W$ is called \emph{differentiable at} $x \in U$, if
there is a bounded linear operator $d_x\phi\in \operatorname{Lin}(V,W)$ such that
%
%e2.5 #&#
\begin{equation}
\label{eq:def-C1} \lim_{h \to0} \frac{\|\phi(x+h) - \phi(x) - d_x\phi(h)\|_W}{\|h\|
_V} = 0.
\end{equation}

In this case, $d_x\phi$ is called the (\emph{total}) \emph{differential of}
$\phi$ \emph{at} $x$. Moreover, $\phi$ is called \emph{continuously
differentiable} or shortly a $C^1$-\emph{map}, if it is differentiable
at every $x \in U$, and the map $d\phi: U \to \operatorname{Lin}(V, W)$, $x \mapsto
d_x\phi$ is continuous. Furthermore, a differentiable map $c: (- \eps
, \eps) \to W$ is called a \emph{curve in} $W$.

%de2.1 #&#
\begin{definition}\label{def:tangentvector}
Let $X \subset V$ be an arbitrary subset and let $x_0 \in X$. Then $v
\in V$ is called a \emph{tangent vector of} $X$ \emph{at} $x_0$, if there is a
curve $c: (- \eps, \eps) \to X \subset V$ such that $c(0) = x_0$ and
$\dot c(0) := d_0c(1) = v$.

The set of all tangent vectors at $x_0$ is called the \emph{tangent
double cone of} $X$ \emph{at} $x_0$ and is denoted by $T_{x_0}X$.
\end{definition}

Since reparametrization of the curve $c$ easily implies that $T_{x_0}X$
is invariant under multiplication by positive or negative scalars, it
is a double cone in $V$. However, for general subsets $X \subset V$,
$T_{x_0}X$ may fail to be a vector subspace, and for $x_0 \neq x_1$,
the tangent cones $T_{x_0}X$ and $T_{x_1}X$ need not be homeomorphic.
We also let
\[
TX := \mathop{\dot\bigcup}_{x_0 \in X} T_{x_0}X \subset X
\times V \subset V \times V,
\]
equipped with the induced topology. Again, $\dot\bigcup$ stands for
the disjoint union of sets. Then $TX$ together with the map $TX \to X$
mapping $T_{x_0}X$ to $x_0$ is a topological fibration, called the
\emph{tangent double cone fibration of} $X$. Since this is a rather bulky
terminology, we shall simply refer to $TX \to X$ as the \emph{tangent
fibration}, but the reader should be aware that, unlike in some texts,
this is not the a synonym for the tangent bundle, as $X$ needs not be a
manifold in general.

If $U \subset V$ is open and $\phi: U \to W$ is a $C^1$-map whose
image is contained in $X \subset W$, then $d_{x_0}\phi(V) \subset
T_{\phi(x_0)}X$, whence $\phi$ induces a continuous map
\[
d\phi: TU = U \times V \longrightarrow TX, \qquad (u, v) \longmapsto
d_u\phi(v).
\]

%th2.1 #&#
\begin{theorem}\label{thm:TM-TP}
Let $\Ss(\Om)$ be the Banach space of signed finite measures on $\Om
$. Then the tangent cones of $\Mm(\Om)$ and $\Pp(\Om)$ at $\mu$
are $T_\mu\Mm(\Om) = \Ss(\Om, \mu)$ and $T_\mu\Pp(\Om) = \Ss
_0(\Om, \mu)$, respectively, so that the tangent fibrations are given as
\begin{equation*}
T\Mm(\Om) = \mathop{\dot\bigcup}_{\mu\in\Mm(\Om)} \Ss(\Om, \mu) \subset
\Mm(\Om) \times\Ss(\Om)
\end{equation*}
and
\begin{equation*}
T\Pp(\Om) = \mathop{\dot\bigcup}_{\mu\in\Pp(\Om)} \Ss_0(\Om,
\mu) \subset\Pp(\Om) \times\Ss(\Om).
\end{equation*}
\end{theorem}

\begin{proof} Let $\nu\in T_{\mu_0} \Mm(\Om)$ and let $(\mu_t)_{t
\in(-\eps, \eps)}$ be a curve in $\Mm(\Om)$ with $\dot\mu_0 =
\nu$. Let $A \subset\Om$ be such that $\mu_0(A) = 0$. Then as $\mu
_t(A) \geq0$, the function $t \mapsto\mu_t(A)$ has a minimum at
$t_0=0$, whence
\[
0 = \left.\frac{d}{dt}\right|_{t=0} \mu_t(A) = \dot
\mu_0(A) = \nu(A),
\]
where the second equation is evident from (\ref{eq:def-C1}). Thus,
$\nu(A) = 0$ whenever $\mu_0(A) = 0$, i.e., $\mu_0$ dominates $\nu
$, so that $\nu\in\Ss(\Om, \mu_0)$. Thus, $T_{\mu_0} \Mm(\Om)
\subset\Ss(\Om, \mu_0)$.

Conversely, given $\nu= \phi\mu_0 \in\Ss(\Om, \mu_0)$, define
$\mu_t := p({\omega}; t) \mu_0$ where
\[
p({\omega}; t) := %
\begin{cases} 1 + t \phi({\omega}) & \mbox{if }t
\phi({\omega}) \geq0,
\\
\exp\bigl(t\phi({\omega})\bigr) & \mbox{if }t\phi({\omega}) < 0. \end{cases}
\]
As $p({\omega}; t) \leq\max(1+t\phi({\omega}), 1)$, it follows
that $\mu_t \in\Mm(\Om)$, and as $|\p_t p({\omega};t)| \leq|\phi
({\omega})| \in L^1(\Om, \mu_0)$ for all $t$, it follows that $t
\mapsto\mu_t$ is a $C^1$-curve in $\Mm(\Om)$ with $\dot\mu_0 =
\phi\mu_0 = \nu$, whence $\nu\in T_{\mu_0}\Mm(\Om)$ as claimed.

To show the statement for $\Pp(\Om)$, let $(\mu_t)_{t \in(-\eps,
\eps)}$ be a curve in $\Pp(\Om)$ with $\dot\mu_0 = \nu$. Then as
$\mu_t$ is a probability measure for all $t$, we conclude
\[
\biggl\|
\int_\Om d\nu\biggr\|= \biggl\|
\int_\Om\frac{1}t \, d(\mu_t -
\mu_0 - t \nu)\biggr\|\leq\frac{\|\mu_t - \mu_0 - t \nu\|
_{\mathrm{TV}}}{|t|} \xrightarrow{t
\to0} 0,
\]
so that $\nu\in\Ss_0(\Om)$. Since $\Pp(\Om) \subset\Mm(\Om)$,
it follows that $T_{\mu_0}\Pp(\Om) \subset T_{\mu_0}\Mm(\Om) \cap
\Ss_0(\Om) = \Ss_0(\Om, \mu_0)$ for all $\mu_0 \in\Pp(\Om)$.

Conversely, given $\nu= \phi\mu_0 \in\Ss_0(\Om, \mu_0)$, define
the curve $\lambda_t := \mu_t\|\mu_t\|_{\mathrm{TV}}^{-1} \in\Pp(\Om)$
with $\mu_t$ from above, which is a $C^1$-curve in $\Pp(\Om)$ as $\|
\mu_t\|_{\mathrm{TV}} > 0$, and it is straightforward that $\lambda_0 = \mu
_0$ and $\dot\lambda_0 = \phi\mu_0 = \nu$.
\end{proof}

%re2.1 #&#
\begin{remark}\label{rem2.1}
\begin{enumerate}
\item[(1)]
Observe that the curves $\mu_t$ and $\lambda_t$ in the proof of
Theorem~\ref{thm:TM-TP} are contained in $\Mm_+(\Om, \mu_0)$ and
$\Pp_+(\Om, \mu_0)$, respectively, whence
\[
T_{\mu}\Mm_+(\Om, \mu_0) = \Ss(\Om, \mu) \quad \mbox{and}
\quad T_{\mu}\Pp_+(\Om, \mu) = \Ss_0(\Om, \mu).
\]
But if $\mu\in\Mm_+(\Om, \mu_0)$, the $\mu$ and $\mu_0$ are
equivalent measures so that $\Ss(\Om, \mu) = \Ss(\Om_,\mu_0) =:
V$ and $\Ss_0(\Om, \mu) = \Ss_0(\Om, \mu_0) =: V_0$. Thus, the
tangent space is the same at all points.

That is, $\Mm_+(\Om, \mu_0) \subset V$ has the property that $T_\mu
\Mm_+(\Om, \mu_0) = V$ for all $\mu$, but $\Mm_+(\Om, \mu_0)
\subset V$ is not an open subset if $\Omega$ is infinite, and the
corresponding statement holds for $\Pp(\Om, \mu_0) \subset\mu_0 +
V_0$. This is a rather unusual phenomenon.
\item[(2)]
The sets $\Pp(\Om)$ and $\Mm(\Om)$ are not Banach submanifolds of
$\Ss(\Om)$, and the tangent fibrations $T\Pp(\Om) \to\Pp(\Om)$
and $T\Mm(\Om) \to\Mm(\Om)$ are not vector bundles, even though
the fibers at each point are closed subspaces. This even fails in the
case where $\Om= \{{\omega}_1, \ldots, {\omega}_k\}$ is finite. In
this case, we may identify $\Ss(\Om)$ with $\R^k$ by the map $\sum_{i=1}^k x_i \delta^{{\omega}_i} \cong(x_1, \ldots, x_k)$, and with this,
\[
T\Mm(\Om) \cong\left\{(x_1, \ldots, x_k;
y_1, \ldots, y_k) \in\R ^k \times
\R^k \, :\, %
\begin{array} {l} x_i \geq0,
\\
x_i = 0 \Rightarrow y_i = 0 \end{array} %
 \right\} \subset\R^{2k},
\]
and this is evidently not a submanifold of $\R^{2k}$. Indeed, in this
case the dimension of $T_\mu\Mm(\Om) = \Ss(\Om, \mu)$ equals $|\{
{\omega}\in\Om\mid\mu({\omega}) > 0\}|$, which varies with $\mu$.
\end{enumerate}
\end{remark}

%%%%%%%%%%%%%%%%%%%%%%%%%%%%%%%%%%%%
%s2.3 #&#
\subsection{Powers of measures} \label{sec:power-density}
%%%%%%%%%%%%%%%%%%%%%%%%%%%%%%%%%%%%

Let us now give the formal definition of powers of measures. On the set
$\Mm(\Om)$ we define the preordering $\mu_1 \leq\mu_2$ if $\mu_2$
dominates $\mu_1$. Then $(\Mm(\Om), \leq)$ is a directed set,
meaning that for any pair $\mu_1, \mu_2 \in\Mm(\Om)$ there is a
$\mu_0 \in\Mm(\Om)$ dominating both of them (use e.g. $\mu_0 :=
\mu_1 + \mu_2$).

For fixed $r \in(0,1]$ and measures $\mu_1 \leq\mu_2$ on $\Om$ we
define the linear embedding
\[
\imath_{\mu_2}^{\mu_1}: L^{1/r}(\Om, \mu_1)
\longrightarrow L^{1/r}(\Om, \mu_2), \qquad \phi\longmapsto
\phi \biggl(\frac
{d\mu_1}{d\mu_2} \biggr)^r.
\]
Observe that
%
%e2.6 #&#
\begin{eqnarray}
\label{eq:norm-Sr} \begin{aligned} \bigl\|\imath_{\mu_2}^{\mu_1}(
\phi)\bigr\|_{L^{1/r}(\Om, \mu_2)} & = \biggl\|
\int _\Om\bigl\|\imath_{\mu_2}^{\mu_1}(
\phi)\bigr\|^{1/r} \,d\mu_2\biggr\|
^r
\\
&= \biggl\|
\int_\Om|\phi|^{1/r} \frac{d\mu_1}{d\mu_2} \,d
\mu_2 \biggr\|^r
\\
& = \biggl\|
\int_\Om|\phi|^{1/r} \,d\mu_1 \biggr
\|^r
\\
&= \|\phi\|_{L^{1/r}(\Om, \mu_1)}, \end{aligned}
\end{eqnarray}
so that $\imath_{\mu_2}^{\mu_1}$ is an isometry. Evidently $\imath
_{\mu_2}^{\mu_1} \imath_{\mu_3}^{\mu_2} = \imath_{\mu_3}^{\mu
_1}$ whenever $\mu_1 \leq\mu_2 \leq\mu_3$. Then we define the \emph{space of $r$th powers of measures on $\Om$} to be the directed limit
over the directed set $(\Mm(\Om), \leq)$
%
%e2.7 #&#
\begin{equation}
\Ss^r(\Om) := \lim_{\longrightarrow} L^{1/r}(\Om,
\mu).
\end{equation}

Let us give a more concrete definition of $\Ss^r(\Om)$. On the
disjoint union of the spaces $L^{1/r}(\Om, \mu)$ for $\mu\in\Mm
(\Om)$ we define the equivalence relation
\begin{eqnarray*}
\begin{aligned}
L^{1/r}(\Om, \mu_1) \ni\phi\sim\psi\in L^{1/r}(
\Om, \mu_2) &\quad \Longleftrightarrow\quad  \imath_{\mu_0}^{\mu_1}(
\phi) = \imath_{\mu_0}^{\mu_2}(\psi)
\\
& \quad \Longleftrightarrow\quad  \phi \biggl(\frac{d\mu_1}{d\mu
_0}
\biggr)^r = \psi \biggl(\frac{d\mu_2}{d\mu_0} \biggr)^r
\end{aligned}
\end{eqnarray*}
for some $\mu_0 \geq\mu_1, \mu_2$. Then $\Ss^r(\Om)$ is the set
of all equivalence classes of this relation.

Denote the equivalence class of $\phi\in L^{1/r}(\Om, \mu)$ by $\phi
\mu^r$, so that $\mu^r \in\Ss^r(\Om)$ is the equivalence class
represented by $1 \in L^{1/r}(\Om, \mu)$. Then the equivalence
relation yields
%
%e2.8 #&#
\begin{equation}
\label{eq:equivalence} \mu_1^r = \biggl( \frac{d\mu_1}{d\mu_2}
\biggr)^r \mu_2^r\qquad \mbox{as elements of
}\Ss^r(\Om)
\end{equation}
whenever $\mu_1 \leq\mu_2$, justifying this notation. In fact, from
this description in the case $r = 1$ we see that
\[
\Ss^1(\Om) = \Ss(\Om).
\]
Observe that by (\ref{eq:norm-Sr}) $\|\phi\|_{L^{1/r}(\Om, \mu)}$ is constant on
equivalence classes, whence there is a norm on $\Ss^r(\Om)$,
denoted by $\|\cdot\|_{1/r}$, for which the inclusions
\[
L^{1/r}(\Om, \mu) \lra\Ss^r(\Om), \qquad \phi\longmapsto
\phi\mu^r
\]
are isometries. For $r = 1$, we have $\|\cdot\|_1 = \|\cdot\|_{\mathrm{TV}}$. Thus,
%
%e2.9 #&#
\begin{equation}
\label{eq:def-normSr} \bigl\|\phi\mu^r\bigr\|_{1/r} =
\|\phi\|_{L^{1/r}(\Om, \mu)} = \biggl\|
\int_\Om|\phi |^{1/r} \, d\mu\biggr\|
^r \qquad \mbox{for }0 < r \leq1.
\end{equation}

Note that the equivalence relation also preserves nonnegativity of
functions, whence we may define the subsets
%
%e2.10 #&#
\begin{eqnarray}
\begin{aligned} \Mm^r(\Om) & := \bigl\{ \phi\mu^r : \mu
\in\Mm(\Om), \phi\geq 0\bigr\},
\\
\Pp^r(\Om) & := \bigl\{ \phi\mu^r : \mu\in\Pp(\Om), \phi
\geq 0, \|\phi \mu^r\|_{1/r} = 1\bigr\}. \end{aligned}
\end{eqnarray}

In analogy to (\ref{def-S(Om,mu)}) we define for a fixed measure $\mu
_0 \in\Mm(\Om)$ and $r \in(0,1]$ the spaces
\begin{eqnarray*}
\Ss^r(\Om, \mu_0) & := & \bigl\{ \phi
\mu_0^r : \phi\in L^{1/r}(\Om,
\mu_0)\bigr\},
\\
\Mm^r(\Om, \mu_0) & := & \bigl\{ \phi
\mu_0^r : \phi\in L^{1/r}(\Om,
\mu_0), \phi\geq0\bigr\},
\\
\Pp^r(\Om, \mu_0) & := & \bigl\{ \phi
\mu_0^r : \phi\in L^{1/r}(\Om,
\mu_0), \phi\geq0, \|\phi \mu_0^r\|_{1/r} = 1\bigr\},
\\
\Ss^r_0(\Om, \mu_0) & := & \biggl\{ \phi
\mu_0^r : \phi\in L^{1/r}(\Om,
\mu_0),
\int_\Om\phi\, d\mu= 0 \biggr\}.
\end{eqnarray*}

The elements of $\Pp^r(\Om, \mu_0)$, $\Mm^r(\Om, \mu_0)$, $\Ss^r(\Om
, \mu_0)$ are said to be \emph{dominated by} $\mu_0^r$.

If $\{ \mu_n \in\Ss(\Om)  :  n \in\N\}$ is a countable family
of (signed) finite measures, then they are dominated by the finite
measure $\mu_0 := \sum_n 2^{-n} \|\nu_n\|_{\mathrm{TV}}^{-1} |\nu_n|$ (cf.
e.g., \cite{Neveu1965}, Ex. IV.1.3). Therefore, any Cauchy sequence
$(\mu_{r;n})_{n \in\N} \in\Ss^r(\Om)$ is contained in $\Ss^r(\Om
, \mu_0)$ for some $\mu_0$. As the embedding $\Ss^r(\Om, \mu_0)
\hookrightarrow\Ss^r(\Om)$ is an isometry, $(\mu_{r;n})_{n \in\N}
\in\Ss^r(\Om, \mu_0)\cong L^{1/r}(\Om, \mu_0)$ is also a Cauchy
sequence and hence convergent.
Thus, $(\Ss^r(\Om), \|\cdot\|_{1/r})$ is a Banach space.

%re2.2 #&#
\begin{remark}\label{rem2.2}
The concept of $r$th powers of measures has been
indicated in \cite{Neveu1965}, Ex. IV.1.4. Moreover, if $\Om$ is a
manifold and $r=1/2$, then $\Ss^{1/2}(\Om)$ is even a Hilbert space
which has been considered in \cite{MR}, Section 6.9.1. In this case, the
diffeomorphism group of $\Om$ acts by isometries on $\Ss^{1/2}(\Om)$
\cite{GS1977}.
\end{remark}

The product of powers of measures can now be defined for all $r, s \in
(0,1)$ with $r + s \leq1$ and for measures $\phi\mu^r \in\Ss^r(\Om
, \mu)$ and $\psi\mu^s \in\Ss^s(\Om, \mu)$:
\begin{equation*}
\bigl(\phi\mu^r\bigr) \cdot\bigl(\psi\mu^s\bigr) := \phi
\psi\mu^{r+s}.
\end{equation*}
By definition $\phi\in L^{1/r}(\Om, \mu)$ and $\psi\in L^{1/s}(\Om
, \mu)$, whence H\"older's inequality implies that $\|\phi
\psi\|_{1/(r+s)} \leq\|\phi\|_{1/r} \|\psi\|_{1/s} < \infty$, so
that $\phi\psi\in L^{1/(r+s)}(\Om, \mu)$ and hence, $\phi\psi\mu
^{r+s} \in\Ss^{r+s}(\Om, \mu)$.
Since by (\ref{eq:equivalence}) this definition of the product is
independent of the choice of representative $\mu$, it follows that it
induces a bilinear product
%
%e2.11 #&#
\begin{equation}
\label{eq:mult-power} \cdot: \Ss^r(\Om) \times\Ss^s(\Om)
\longrightarrow\Ss^{r+s}(\Om ),\qquad \mbox{where }r, s, r + s \in(0,1],
\end{equation}
satisfying the H\"older inequality
\begin{equation*}
\|\nu_r \cdot\nu_s\|_{1/(r+s)} \leq\|
\nu_r\|_{1/r} \|\nu_s\|_{1/s},
\end{equation*}
so that the product in (\ref{eq:mult-power}) is a bounded bilinear map.

In analogy to Theorem~\ref{thm:TM-TP}, we can also determine the
tangent fibrations of the subsets $\Pp^r(\Om) \subset\Mm^r(\Om)
\subset\Ss^r(\Om)$.

%pr2.1 #&#
\begin{proposition}\label{prop:TMr-SMr}
For each $\mu\in\Mm(\Om)$ ($\mu\in\Pp(\Om)$, respectively), the
tangent cones of $\Pp^r(\Om) \subset\Mm^r(\Om) \subset\Ss^r(\Om
)$ at $\mu^r$ are $T_{\mu^r} \Mm^r(\Om) = \Ss^r(\Om, \mu)$ and
$T_{\mu^r} \Pp^r(\Om) = \Ss^r_0(\Om, \mu)$, respectively, so that
the tangent fibrations are given as
\begin{equation*}
T\Mm^r(\Om) = \mathop{\dot\bigcup}_{\mu_r \in\Mm^r(\Om)} \Ss
^r(\Om, \mu) \subset\Mm^r(\Om) \times\Ss^r(
\Om)
\end{equation*}
and
\begin{equation*}
T\Pp^r(\Om) = \mathop{\dot\bigcup}_{\mu_r \in\Pp^r(\Om)} \Ss
^r_0(\Om, \mu) \subset\Pp^r(\Om) \times
\Ss^r(\Om).
\end{equation*}
\end{proposition}

\begin{proof}
We have to adapt the proof of Theorem~\ref{thm:TM-TP}.
The proof of the statements $\Ss^r(\Om, \mu) \subset T_{\mu^r}\Mm
^r(\Om)$ and $\Ss_0^r(\Om, \mu) \subset T_{\mu^r}\Pp^r(\Om)$ is
identical to that of the corresponding statement in Theorem~\ref
{thm:TM-TP}; just as in that case, one shows that for $\phi\in
L^{1/r}(\Om, \mu_0)$ the curves $\mu_t^r := p({\omega}; t)\mu_0^r$
with $p({\omega};t) := 1 + t\phi({\omega})$ if $t\phi({\omega})
\geq0$ and $p({\omega};\xi) = \exp(t\phi({\omega}))$ if $t\phi
({\omega}) < 0$ is a differentiable curve in $\Mm^r(\Om)$, and
$\lambda_t^r := \mu_t^r/ \|\mu_t^r\|_{1/r}$ is a differentiable
curve in $\Pp^r(\Om)$, and their derivative is $\phi\mu_0^r$ at $t=0$.

In order to show the other direction, let $(\mu_t^r)_{t \in(-\eps,
\eps)}$ be a curve in $\Mm^r(\Om)$. Since there is a measure $\hat
\mu$ dominating the countable family $(\mu_t^r)_{t \in\Q\cap(-\eps
, \eps)}$ and since $\Ss^r(\Om, \hat\mu) \subset\Ss^r(\Om)$ is
closed, it follows that $\mu_t^r \in\Mm(\Om, \hat\mu)$ for all
$t$. Now we can apply the argument from the proof of Theorem~\ref
{thm:TM-TP} to the curve $t \mapsto(\mu_t^r \cdot\hat\mu
^{1-r})(A)$ for $A \subset\Om$.
\end{proof}

Besides multiplication of roots of measures, we also wish to take their
powers. Here, we have two possibilities to deal with signs. For $0 < k
\leq r^{-1}$ and $\nu_r = \phi\mu^r \in\Ss^r(\Om)$ we define
\begin{equation*}
|\nu_r|^k := |\phi|^k \mu^{rk}
\quad \mbox{and} \quad \tilde\nu _r^k := \sign(\phi) |
\phi|^k \mu^{rk}.
\end{equation*}
Since $\phi\in L^{1/r}(\Om, \mu)$, it follows that $|\phi|^k \in
L^{1/rk}(\Om, \mu)$, so that $|\nu_r|^k, \tilde\nu_r^k \in\Ss
^{rk}(\Om)$. By (\ref{eq:equivalence}) these powers are well defined,
independent of the choice of the measure $\mu$, and, moreover,
%
%e2.12 #&#
\begin{equation}
\label{eq:power-norm} \bigl\||\nu_r|^k \bigr\|
_{1/(rk)} = \bigl\|\tilde\nu_r^k\bigr\|
_{1/(rk)} = \|\nu_r\|_{1/r}^k.
\end{equation}

%pr2.2 #&#
\begin{proposition} \label{prop:powers-C1}
Let $r \in(0,1]$ and $ 0 < k \leq1/r$, and consider the maps
\[
\pi^k, \tilde\pi^k: \Ss^r(\Om)
\longrightarrow\Ss^{rk}(\Om),\qquad %
\begin{array} {l}
\pi^k(\nu) := |\nu|^k,
\\
\tilde\pi^k(\nu) := \tilde\nu^k. \end{array} %
\]
Then $\pi^k$, $\tilde\pi^k$ are continuous maps. Moreover, for $1 < k
\leq1/r$ they are $C^1$-maps between Banach spaces, and their
derivatives are given as
%
%e2.13 #&#
\begin{equation}
\label{eq:form-dpi} d_{\nu_r}\tilde\pi^k(\rho_r) = k |
\nu_r|^{k-1} \cdot\rho_r\quad \mbox{and}\quad
d_{\nu_r} \pi^k (\rho_r) = k \tilde\nu
_r^{k-1} \cdot\rho_r.
\end{equation}
\end{proposition}

Observe that for $k = 1$, $\pi^1(\nu_r) = |\nu_r|$ fails to be
$C^1$, whereas $\tilde\pi^1(\nu_r)= \nu_r$, so that $\tilde\pi^1$
is the identity and hence a $C^1$-map.

\begin{proof}[Proof of Proposition~\ref{prop:powers-C1}]
Let us first assume that $0 < k \leq1$. We assert that{proof}
in this case, there are constants $C_k$, $\tilde C_k > 0$ such that for
all $x, y \in\R$
%
%e2.14 #&#
\begin{equation}
\label{eq:p-estimate} \begin{aligned} &\bigl\||x+y|^k-|x|^k
\bigr\|\leq C_k |y|^k\quad \mbox{and}
\\
&\bigl\|\sign(x+y) |x+y|^k - \sign(x) |x|^k\bigr
\|\leq\tilde C_k |y|^k. \end{aligned}
\end{equation}
Namely, by homogeneity it suffices to show this for $y = 1$, and since
the functions
\[
x \longmapsto|x+1|^k-|x|^k\quad \mbox{and} \quad x
\longmapsto \sign(x+1) |x+1|^k - \sign(x) |x|^k
\]
are continuous and have finite limits for $x \to\pm\infty$, it
follows that they are bounded, showing (\ref{eq:p-estimate}).

Let $\nu_1, \nu_2 \in\Ss^r(\Om)$, and choose $\mu_0 \in\Mm(\Om
)$ such that $\nu_1, \nu_2 \in\Ss^r(\Om, \mu_0)$, i.e., $\nu_i =
\phi_i \mu_0^r$ with $\phi_i \in L^{1/r}(\Om, \mu_0)$. Then\vspace*{-1pt}
\begin{eqnarray*}
\bigl\|\pi^k(\nu_1 + \nu_2) -
\pi^k(\nu_1)\bigr\|_{1/(rk)} & = & \bigl
\||\phi_1 + \phi_2|^k - |
\phi_1|^k\bigr\|_{1/(rk)}
\\
& \leq& C_k \bigl\||\phi_2|^k\bigr
\|_{1/rk} \qquad \mbox{by (\ref
{eq:p-estimate})}
\\
& = & C_k \|\nu_2\|_{1/r}^k
\qquad \mbox{by (\ref{eq:power-norm})},
\end{eqnarray*}
so that $\lim_{\|\nu_2\|_{1/r} \to0} \|\pi^k(\nu_1 + \nu_2) - \pi
^k(\nu_1)\|_{1/(rk)} = 0$, showing the continuity of $\pi^k$ for \mbox{$0 <
k \leq1$}. The continuity of $\tilde\pi^k$ follows analogously.

Now let us assume that $1 < k \leq1/r$. In this case, the functions\vspace*{-1pt}
\[
x \longmapsto|x|^k \quad \mbox{and} \quad x \longmapsto\sign(x)
|x|^k
\]
with $x \in\R$ are $C^1$-maps with respective derivatives\vspace*{-1pt}
\[
x \longmapsto k \sign(x) |x|^{k-1}\quad \mbox{and}\quad x \longmapsto k
|x|^{k-1}.
\]

Thus, if we pick $\nu_i = \phi_i \mu_0^r$ as above, then by the mean
value theorem we have\vspace*{-1pt}
\begin{eqnarray*}
\pi^k(\nu_1 + \nu_2) - \pi^k(
\nu_1) & = & \bigl(\|\phi_1 + \phi_2
\|^k - \|\phi_1\|^k\bigr)
\mu_0^{rk}
\\
& = & k \sign(\phi_1 + \eta\phi_2) \|
\phi_1 + \eta\phi_2\|^{k -
1}
\phi_2 \mu_0^{rk}
\\
& = & k \sign(\phi_1 + \eta\phi_2) \|
\phi_1 + \eta\phi_2\|^{k -
1}
\mu_0^{r(k-1)} \cdot\nu_2
\end{eqnarray*}
for some function $\eta: \Om\to(0,1)$. If we let $\nu_\eta:= \eta
\phi_2 \mu_0^r$, then $\|\nu_\eta\|_{1/r} \leq\|\nu_2\|_{1/r}$,
and we get\vspace*{-1pt}
\[
\pi^k(\nu_1 + \nu_2) - \pi^k(
\nu_1) = k \tilde\pi^{k-1}(\nu_1 +
\nu_\eta) \cdot\nu_2.
\]
With the definition of $d_{\nu_1}\tilde\pi^k$ from (\ref
{eq:form-dpi}) we have\vspace*{-1pt}
\begin{align*}
&\bigl\|\pi^k(\nu_1 + \nu_2) -
\pi^k(\nu_1) - d_{\nu_1} \pi^k(\nu
_2)\bigr\|_{1/(rk)}
\\
&\qquad = \bigl\|k \bigl(\tilde\pi^{k-1}(\nu_1 +
\nu_\eta) - \tilde\pi ^{k-1}(\nu_1)\bigr) \cdot
\nu_2\bigr\|_{1/(rk)}
\\
&\qquad \leq k \bigl\|\tilde\pi^{k-1}(\nu_1 +
\nu_\eta) - \tilde\pi ^{k-1}(\nu_1)\bigr\|
_{1/(r(k-1))} \|\nu_2\|_{1/r}
\end{align*}
and hence,\vspace*{-1pt}
\[
\frac{\|\pi^k(\nu_1 + \nu_2) - \pi^k(\nu_1) - d_{\nu_1} \pi
^k(\nu_2)\|_{\frac{1}{rk}}}{\|\nu_2\|_{\frac{1}r}} \leq k \bigl\|\tilde\pi ^{k-1}(
\nu_1 + \nu_\eta) - \tilde\pi^{k-1}(
\nu_1)\bigr\|_{\frac{1}{r(k-1)}}.
\]
Thus, the differentiability of $\pi^k$ will follow if\vspace*{-1pt}
\[
\bigl\|\tilde\pi^{k-1}(\nu_1 + \nu_\eta) -
\tilde\pi^{k-1}(\nu_1)\bigr\|_{1/(r(k-1))}
\xrightarrow{\|\nu_2\|_{1/r} \to0} 0,
\]
and because of $\|\nu_\eta\|_{1/r} \leq\|\nu_2\|_{1/r}$, this is
the case if $\tilde\pi^{k - 1}$ is continuous.\vadjust{\goodbreak}

Analogously, one shows that $\tilde\pi^k$ is differentiable if $\pi
^{k-1}$ is continuous.

Since we already know continuity of $\pi^k$ and $\tilde\pi^k$ for $0
< k \leq1$, and since $C^1$-maps are continuous, the claim now follows
by induction on $\lceil k \rceil$.
\end{proof}

Thus, (\ref{eq:form-dpi}) implies that the differentials of $\pi^k$
and $\tilde\pi^k$ (which coincide on $\Pp^r(\Om)$ and $\Mm^r(\Om
)$) yield continuous maps
\begin{equation*}
d\pi^k = d\tilde\pi^k\, :\, %
\begin{array} {l}
T\Pp^r(\Om) \longrightarrow T\Pp^{rk}(\Om),
\\
T\Mm^r(\Om) \longrightarrow T\Mm^{rk}(\Om), \end{array}
\qquad %
(\mu, \rho) \longmapsto k \mu^{rk-r} \cdot\rho.
\end{equation*}

%%%%%%%%%%%%%%%%%%%%%%%%%%%%%%%%%%%%
%s3 #&#
\section{Congruent embeddings} \label{sec:congruent}
%%%%%%%%%%%%%%%%%%%%%%%%%%%%%%%%%%%%

%%%%%%%%%%%%%%%%%%%%%%%%%%%%%%%%%%%%
%s3.1 #&#
\subsection{Statistics and congruent embeddings}\label{sec3.1}
%%%%%%%%%%%%%%%%%%%%%%%%%%%%%%%%%%%%

Given two measurable sets $\Omega$ and $\Omega'$, a measurable map
\[
\kappa:\Omega\longrightarrow\Omega'
\]
will be called a \emph{statistic}.
Any (signed) measure $\mu$ on $\Omega$, induces a (signed) measure
$\kappa_{*}\mu$ on $\Omega'$, via
%
%e3.1 #&#
\begin{equation}
\label{eq:push-forward} \kappa_{*}\mu(A) := \mu\bigl(\kappa^{-1}A
\bigr),
\end{equation}
which is called the \emph{push-forward of} $\mu$ \emph{by} $\kappa$.
Note that
%
%e3.2 #&#
\begin{equation}
\label{eq:push-fwd-kappa} \kappa_*: \Ss(\Om) \longrightarrow\Ss\bigl(\Om'
\bigr)
\end{equation}
is a bounded linear map which is monotone, that is, it maps nonnegative
measures to nonnegative measures. When using the Jordan decomposition
(\ref{Jordan-dec}), we obtain
\[
\|\kappa_*\mu\|_{\mathrm{TV}} = \|\kappa_*\mu_+ - \kappa_*
\mu_-\|\bigl(\Om'\bigr) \leq\kappa_*\mu_+\bigl(
\Om'\bigr) + \kappa_*\mu_-\bigl(\Om'\bigr) = |\mu|(
\Om) = \|\mu\|_{\mathrm{TV}}.
\]
Thus,
%
%e3.3 #&#
\begin{equation}
\label{eq:mu-abs} \|\kappa_*\mu\|_{\mathrm{TV}} \leq\|\mu\|_{\mathrm{TV}} \quad
\mbox{with equality iff} \quad \kappa_*\mu_+ \perp\kappa_*\mu_-.
\end{equation}
In particular, $\kappa_*$ preserves the total variation of nonnegative
measures, and whence maps probability measures to probability measures, i.e.
\[
\kappa_*\bigl(\Pp(\Om)\bigr) \subset\Pp\bigl(\Om'\bigr).
\]
Furthermore, if $\mu_1$ dominates $\mu_2$, then $\kappa_*\mu_1$
dominates $\kappa_*\mu_2$ by (\ref{eq:push-forward}), whence $\kappa
_*$ yields bounded linear maps
%
%e3.4 #&#
\begin{equation}
\label{eq:kappa-1} \kappa_*: \Ss(\Om, \mu) \longrightarrow\Ss\bigl(
\Om', \kappa_*\mu\bigr),
\end{equation}
and if we write
%
%e3.5 #&#
\begin{equation}
\label{eq:cond-exp} \kappa_*(\phi\mu) = \phi' \kappa_*\mu,
\end{equation}
then $\phi' \in L^1(\Om', \kappa_*\mu)$ is called the \emph{conditional expectation of} $\phi\in L^1(\Om, \mu)$ \emph{given} $\kappa$.
This yields a bounded linear map
%
%e3.6 #&#
\begin{equation}
\label{eq:cond-ex-L1} \kappa_*^\mu: L^1(\Om, \mu)
\longrightarrow L^1\bigl(\Om', \mu'\bigr),\qquad
\phi\longmapsto\phi'
\end{equation}
with $\phi'$ from (\ref{eq:cond-exp}).

We also define the pull-back of a measurable function $\phi': \Om'
\to\R$ as
\[
\kappa^*\phi' := \phi' \circ\kappa.
\]
If $A' \subset\Om'$ and $A := \kappa^{-1}(A')$ we have $\chi_A =
\kappa^*\chi_{A'}$, and thus, (\ref{eq:push-forward}) is equivalent
to $\chi_{A'}\kappa_*\mu= \kappa_*(\chi_A \mu) = \kappa_*(\kappa
^*\chi_{A'} \mu)$, and by linearity and the density of step functions
in $L^1(\Om', \kappa_*\mu)$ this implies for $\phi' \in L^1(\Om',
\kappa_*\mu)$
%
%e3.7 #&#
\begin{equation}
\label{eq:pushforward-mult} \kappa_*\bigl(\kappa^*\phi'\mu\bigr) =
\phi' \kappa_*\mu \quad \mbox {or, equivalently,} \quad
\kappa_*^\mu\bigl(\kappa^*\phi'\bigr) =
\phi'.
\end{equation}

Recall that $\Mm(\Om)$ and $\Ss(\Om)$ denote the spaces of \emph{all} (signed) measures on $\Om$, whereas $\Mm(\Om, \mu)$ and $\Ss(\Om
, \mu)$ denote the subspaces of the (signed) measures on $\Om$ which
are dominated by $\mu$.

%de3.1 #&#
\begin{definition}[Congruent embedding]\label{def:cong-emb}
Let $\kappa: \Om\to\Om'$ be a statistic and $\mu' \in\Mm(\Om
')$. A $\kappa$-\emph{congruent embedding} is a bounded linear map
$K_*: \Ss(\Om', \mu') \to\Ss(\Om)$ such that:
\begin{enumerate}
\item[(1)]
$K_*$ is monotone, i.e., it maps nonnegative measures to nonnegative
measures, or shortly: $K_*(\Mm(\Om', \mu')) \subset\Mm(\Om)$.
\item[(2)]
$\kappa_* K_* (\nu') = \nu'$ for all $\nu' \in\Ss(\Om', \mu')$.
\end{enumerate}

Furthermore, the image of a $\kappa$-congruent embedding $K_*$ in $\Ss
(\Om)$ is called a $\kappa$-\emph{congruent subspace of} $\Ss(\Om)$.
\end{definition}

%ex3.1 #&#
\begin{example} \label{ex:cong-emb}
Let $\kappa: \Om\to\Om'$ be a statistic, let $\mu\in\Mm(\Om)$
and $\mu' := \kappa_*\mu\in\Mm(\Om')$. Then the map
%
%e3.8 #&#
\begin{equation}
\label{congruent-mu} K_\mu: \Ss\bigl(\Om', \mu'
\bigr) \longrightarrow\Ss(\Om, \mu) \subset\Ss (\Om), \qquad \phi'
\mu' \longmapsto\kappa^*\phi' \mu
\end{equation}
for all $\phi' \in L^1(\Om', \mu')$ is a $\kappa$-congruent
embedding, since
\[
\kappa_*\bigl(K_\mu\bigl(\phi' \mu'\bigr)
\bigr) = \kappa_*\bigl( \kappa^*\phi' \mu\bigr) \stackrel{
\text{(\ref{eq:pushforward-mult})}}= \phi' \kappa_*\mu= \phi'
\mu'.
\]
\end{example}

We shall now see that the above example exhausts \emph{all}
possibilities of congruent embeddings.

%pr3.1 #&#
\begin{proposition} \label{prop:cong-emb}
Let $\kappa: \Om\to\Om'$ be a statistic, let $K_*: \Ss(\Om', \mu
') \to\Ss(\Om)$ for some $\mu' \in\Mm(\Om')$ be a $\kappa
$-congruent embedding, and let $\mu:= K_*\mu' \in\Mm(\Om)$.

Then $K_* = K_\mu$ with the map $K_\mu$ given in (\ref{congruent-mu}).
\end{proposition}

\begin{proof}
We have to show that $K_*(\phi' \mu') = \kappa^*\phi' \mu$ for all
$\phi' \in L^1(\Om', \mu')$. By continuity, it suffices to show this
for step functions, as these are dense in $L^1(\Om', \mu')$, whence
by linearity, we have to show that for all $A' \subset\Om'$, $A :=
\kappa^{-1}(A') \subset\Om$\vspace*{-1pt}
%
%e3.9 #&#
\begin{equation}
\label{eq:K-pullback-step} K_*\bigl(\chi_{A'} \mu'\bigr) =
\chi_A \mu.
\end{equation}

Let $A'_1 := A'$ and $A'_2 = \Om' \backslash A'$, and let $A_i :=
\kappa^{-1}(A_i')$. We define the measures $\mu'_i := \chi_{A_i'}
\mu' \in\Mm(\Om')$, and $\mu_i := K_*\mu'_i \in\Mm(\Om)$.
Observe that the monotonicity of $K_*$ implies that $\mu_i$ are indeed
(nonnegative) measures. Since $\mu_1' + \mu_2' = \mu'$, it follows
that $\mu_1 + \mu_2 = \mu$ by the linearity of $K_*$.

Taking indices $\bmod 2$, and using $\kappa_*\mu_i = \kappa_*K_*\mu
_i' = \mu_i'$ by the $\kappa$-congruency of $K_*$, note that\vspace*{-1pt}
\[
\mu_i(A_{i+1}) = \mu_i\bigl(
\kappa^{-1}\bigl(A'_{i+1}\bigr)\bigr) = \kappa_*
\mu _i\bigl(A'_{i+1}\bigr) =
\mu'_i\bigl(A'_{i+1}\bigr) = 0.
\]
Thus, for any measurable $B \subset\Om$ we have\vspace*{-1pt}
\begin{eqnarray*}
\mu_1(B) & = & \mu_1(B \cap A_1) \qquad
\mbox{since } \mu_1(B \cap A_2) \leq\mu_1(A_2)
= 0
\\
& = & \mu_1(B \cap A_1) + \mu_2(B \cap
A_1) \qquad \mbox{since } \mu_2(B \cap A_1)
\leq\mu_2(A_1)= 0
\\
& = & \mu(B \cap A_1)\qquad \mbox{since } \mu= \mu_1 +
\mu_2
\\
& = & (\chi_A \mu) (B) \qquad \mbox{since } A_1 = A.
\end{eqnarray*}
That is, $\chi_A \mu= \mu_1 = K_*\mu_1' = K_*(\chi_{A'} \mu')$,
so that (\ref{eq:K-pullback-step}) follows.
\end{proof}

%%%%%%%%%%%%%%%%%%%%%%%%%%%%%%%%%%%%
%s3.2 #&#
\subsection{Markov kernels and Markov morphisms}
%%%%%%%%%%%%%%%%%%%%%%%%%%%%%%%%%%%%

%de3.2 #&#
\begin{definition}[Markov kernel and Markov morphism,
cf. \cite{AJLS,Chentsov1965,MS1966}] \label{def:markov}
A \emph{Markov kernel} between two measurable spaces $\Omega$ and
$\Omega'$ is a map $K: \Om\to\Pp(\Om')$, associating to each
${\omega}\in\Om$ a probability measure on $\Om'$ such that for each
fixed measurable $A' \subset\Om'$ the map\vspace*{-1pt}
\[
\Om\longrightarrow[0,1],\qquad {\omega}\longmapsto K\bigl({\omega};A'
\bigr) := K({\omega}) \bigl(A'\bigr)
\]
is measurable. The \emph{Markov morphism induced by $K$} is the linear map\vspace*{-1pt}
%
%e3.10 #&#
\begin{equation}
\label{eq:Markov-linear} K_*: \Ss(\Om) \longrightarrow\Ss(\Om),\qquad K_*\mu
\bigl(A'\bigr) :=
\int _{\Om} K\bigl({\omega}; A'\bigr) \, d\mu({
\omega}).
\end{equation}
\end{definition}

We shall use the notation $K_*(\mu;A') := K_*\mu(A')$. Since
$K({\omega}) \in\Pp(\Om')$, it follows that $K({\omega}; \Om') =
1$ and hence (\ref{eq:Markov-linear}) implies that $K_*\mu(\Om') =
\mu(\Om)$. Thus,\vspace*{-1pt}
%
%e3.11 #&#
\begin{equation}
\label{eq:markov-isometry-M} \|K_*\mu\|_{\mathrm{TV}} = \|\mu\|
_{\mathrm{TV}}\qquad \mbox{ for all $\mu\in\Mm (\Om)$}.
\end{equation}

In particular, a Markov morphism maps probability measures to
probability measures. For a general measure $\mu\in\Ss(\Om)$, (\ref
{Jordan-dec2}) implies that $|K_*(\mu; A')| \leq K_*(|\mu|; A')$ for
all $A'$ and hence,\vspace*{-1pt}
\[
\|K_*\mu\|_{\mathrm{TV}}\leq\bigl\|K_*\|\mu\|
\bigr\|_{\mathrm{TV}} = \|\mu\|_{\mathrm{TV}} \qquad \mbox
{for all }\mu\in\Ss(\Om),
\]
so that $K_*: \Ss(\Om) \to\Ss(\Om')$ is a bounded linear map.\vadjust{\goodbreak}

Observe that we can recover the Markov kernel $K$ from $K_*$ using the relation
\begin{equation*}
K({\omega}) = K_*\delta^{{\omega}}\qquad \mbox{for all } {\omega }\in\Om,
\end{equation*}
where $\delta^{{\omega}}$ denotes the Dirac measure supported at
${\omega}\in\Om$.

%re3.1 #&#
\begin{remark}\label{rem3.1}
From (\ref{eq:Markov-linear}) it is immediate that $K_*$ preserves
dominance of measures, i.e., if $\mu$ dominates $\tilde\mu$, then
$K_*\mu$ dominates $K_*\tilde\mu$. Thus, for each $\mu\in\Mm(\Om
)$ there is a restriction
\begin{equation*}
K_*: \Ss(\Om, \mu) \longrightarrow\Ss\bigl(\Om',
\mu'\bigr),
\end{equation*}
where $\mu' := K_*\mu$. This again induces a bounded linear map
%
%e3.12 #&#
\begin{equation}
\label{eq:defK*mu} K_*^\mu: L^1(\Om, \mu) \longrightarrow
L^1\bigl(\Om', \mu'\bigr),\qquad \phi
\longmapsto\phi',
\end{equation}
where $\phi'$ is given by
%
%e3.13 #&#
\begin{equation}
\label{eq:cond-K} K_*(\phi\mu) = \phi' \mu',
\end{equation}
and as for statistics, $\phi'$ is called the \emph{conditional
expectation of} $\phi$ \emph{given} $K$, cf. (\ref{eq:cond-exp}).
\end{remark}

%de3.3 #&#
\begin{definition}[Composition of Markov kernels]\label{def3.3}
Let $\Om_i$, $i = 1,2,3$ be measurable spaces, and let $K_i: \Om_i
\to\Pp(\Om_{i+1})$ for $i = 1,2$ be Markov kernels. The \emph{composition of} $K_1$ \emph{and} $K_2$ is the Markov kernel
\[
K_2 K_1: \Om_1 \longrightarrow\Pp({
\Om_3}), \qquad {\omega }\longmapsto(K_2)_*
\bigl(K_1({\omega})\bigr).
\]
\end{definition}

Since $\|(K_2)_*(K_1({\omega}))\|_{\mathrm{TV}}= \|K_1({\omega})\|_{\mathrm{TV}} = 1$
by (\ref{eq:markov-isometry-M}), $(K_2)_*(K_1({\omega}))$ is a
probability measure, hence this composition yields indeed a Markov
kernel. Moreover, it is straightforward to verify that this composition
is associative, and for the induced Markov morphism we have
%
%e3.14 #&#
\begin{equation}
\label{eq:Markov-chainrule} (K_2 K_1)_* = (K_2)_*
(K_1)_*.
\end{equation}

Markov kernels are generalizations of statistics. In fact, a statistic
$\kappa: \Om\to\Om'$ induces a Markov kernel by
\begin{equation*}
K^\kappa({\omega}) := \delta^{\kappa({\omega})}, \quad \mbox{so that} \quad
K^\kappa\bigl({\omega}; A'\bigr) := \chi_{\kappa
^{-1}(A')}({
\omega}).
\end{equation*}
In this case, the Markov morphism induced by $K^\kappa$ is the map
$\kappa_*: \Ss(\Om) \to\Ss(\Om')$ from (\ref
{eq:push-fwd-kappa}). We shall write the Markov kernel $K^\kappa$ also
as $\kappa$ if there is no danger of confusion.

%de3.4 #&#
\begin{definition}[Congruent Markov kernels]\label{def3.4}
A Markov kernel $K: \Om' \to\Pp(\Om)$ is called $\kappa
$-\emph{congruent} for a statistic $\kappa: \Om\to\Om'$ if
%
%e3.15 #&#
\begin{equation}
\label{eq:Cong-Mar-ker} \kappa_*K\bigl({\omega}'\bigr) =
\delta^{{\omega}'} \qquad \mbox{for all } {\omega}' \in
\Om',
\end{equation}
or, equivalently,
\[
\bigl(K^\kappa K\bigr)_* = \mathrm{Id}_{\Ss(\Om')}: \Ss\bigl(
\Om'\bigr) \longrightarrow\Ss\bigl(\Om'\bigr).
\]
In this case, we also call the induced Markov morphism $K_*: \Ss(\Om
') \to\Ss(\Om)$ $\kappa$-congruent.\vadjust{\goodbreak}
\end{definition}

In order to relate the notions of $\kappa$-congruent Markov morphism
and $\kappa$-congruent embeddings from Definition~\ref{def:cong-emb},
we need the notion of $\kappa$-transverse measures.

%de3.5 #&#
\begin{definition}[Transverse measures]\label{def3.5}
Let $\kappa: \Om\to\Om'$ be a statistic. A measure $\mu\in\Mm
(\Om)$ is said to admit \emph{$\kappa$-transverse measures} if there
are measures $\mu_{{\omega}'}^\perp$ on $\kappa^{-1}({\omega}')$
such that for all $\phi\in L^1(\Om, \mu)$
%
%e3.16 #&#
\begin{equation}
\label{eq:mu-transverse}
\int_\Om\phi \, d\mu=
\int_{\Om'} \biggl(
\int_{\kappa
^{-1}({\omega}')} \phi\, d\mu_{{\omega}'}^\perp \biggr)\,
d\mu '\bigl({\omega}'\bigr),
\end{equation}
where $\mu' := \kappa_*\mu$. In particular, the function
\[
\Om' \longrightarrow\hat\R, \qquad {\omega}'
\longmapsto
\int _{\kappa^{-1}({\omega}')} \phi\,  d\mu_{{\omega}'}^\perp
\]
is measurable for all $\phi\in L^1(\Om, \mu)$.
\end{definition}

Observe that the choice of $\kappa$-transverse measures $\mu_{{\omega
}'}^\perp$ is not unique, but rather, one can change these measures
for all ${\omega}'$ in a $\mu'$-null set.

%pr3.2 #&#
\begin{proposition} \label{prop:transverse-probab}
Let $\kappa: \Om\to\Om'$ be a statistic and $\mu\in\Mm(\Om)$ a
measure which admits $\kappa$-transverse measures $\{\mu_{{\omega
}'}^\perp :  {\omega}' \in\Om'\}$. Then $\mu_{{\omega}'}^\perp
$ is a probability measure for almost every ${\omega}' \in\Om'$ and
hence, we may assume w.l.o.g. that $\mu_{{\omega}'}^\perp\in\Pp
(\kappa^{-1}({\omega}'))$ for \emph{all} ${\omega}' \in\Om'$.
\end{proposition}

\begin{proof}
Given $\eps> 0$, define $A_\eps' := \{{\omega}' \in\Om' :   \mu
_{{\omega}'}^\perp(\kappa^{-1}({\omega}')) \geq1+\eps\}$. Then
for $\phi:= \chi_{\kappa^{-1}(A_\eps')}$ the two sides of equation
(\ref{eq:mu-transverse}) read
\begin{eqnarray*}
\int_\Om\chi_{\kappa^{-1}(A_\eps')} \, d\mu& = & \mu\bigl(\kappa
^{-1}\bigl(A_\eps'\bigr)\bigr) =
\mu'\bigl(A_\eps'\bigr)
\\
\int_{\Om'} \biggl(
\int_{\kappa^{-1}({\omega}')} \chi_{\kappa
^{-1}(A_\eps)} \, d\mu_{{\omega}'}^\perp
\biggr) \, d\mu'\bigl({\omega }'\bigr) & = &
\int_{A'_\eps} \biggl(
\int_{\kappa^{-1}(\omega')} d\mu _{{\omega}'}^\perp \biggr)\,  d
\mu'\bigl({\omega}'\bigr)
\\
& = &
\int_{A'_\eps} \mu_{{\omega}'}^\perp\bigl(
\kappa^{-1}\bigl(\omega '\bigr)\bigr) \, d\mu'
\bigl({\omega}'\bigr)
\\
& \geq& (1 + \eps) \mu'\bigl(A_\eps'
\bigr).
\end{eqnarray*}
Thus, (\ref{eq:mu-transverse}) implies
\[
\mu'\bigl(A_\eps'\bigr) \geq(1 + \eps)
\mu'\bigl(A_\eps'\bigr),
\]
and hence, $\mu'(A_\eps') = 0$ for all $\eps> 0$. Thus,
\[
\mu'\bigl(\bigl\{{\omega}' \in\Om' :
\mu_{{\omega}'}^\perp\bigl(\kappa ^{-1}\bigl({
\omega}'\bigr)\bigr) > 1\bigr\}\bigr) = \mu' \Biggl(
\bigcup_{n=1}^\infty A_{1/n}'
\Biggr) \leq\sum_{n=1}^\infty
\mu'\bigl(A_{1/n}'\bigr) = 0,
\]
whence $\{{\omega}' \in\Om' :   \mu_{{\omega}'}^\perp(\kappa
^{-1}({\omega}')) > 1\}$ is a $\mu'$-null set. Analogously, $\{
{\omega}' \in\Om' :   \mu_{{\omega}'}^\perp(\kappa
^{-1}({\omega}')) < 1\}$ is a $\mu'$-null set, that is, $\mu
_{{\omega}'}^\perp\in\Pp(\kappa^{-1}({\omega}'))$ and hence $\|
\mu_{{\omega}'}^\perp\|_{\mathrm{TV}} = 1$ for $\mu'$-a.e. $\mu' \in\Om'$.

Thus, if we replace $\mu_{{\omega}'}^\perp$ by $\tilde\mu_{{\omega
}'}^\perp:= \|\mu_{{\omega}'}^\perp\|_{\mathrm{TV}}^{-1} \mu_{{\omega
}'}^\perp$, then $\tilde\mu_{{\omega}'}^\perp\in\Pp(\kappa
^{-1}({\omega}'))$ for \emph{all} ${\omega}' \in\Om'$, and since
$\tilde\mu_{{\omega}'}^\perp= \mu_{{\omega}'}^\perp$ for $\mu
'$-a.e. ${\omega}' \in\Om'$, it follows that (\ref
{eq:mu-transverse}) holds when replacing $\mu_{{\omega}'}^\perp$ by
$\tilde\mu_{{\omega}'}^\perp$.
\end{proof}

We are now ready to relate the notions of $\kappa$-congruent
embeddings and $\kappa$-congruent Markov kernels.

%th3.1 #&#
\begin{theorem} \label{Markov-congruent}
Let $\kappa: \Om\to\Om'$ be a statistic and $\mu' \in\Mm(\Om')$
be a measure.
\begin{enumerate}
\item[(1)]
If $K: \Om' \to\Pp(\Om)$ is a $\kappa$-congruent Markov kernel,
then the restriction of
$K_*$ to $\Ss(\Om', \mu') \subset\Ss(\Om')$ is a $\kappa
$-congruent embedding and hence, for $\phi' \in L^1(\Om', \mu')$ we have
\begin{equation*}
K_*\bigl(\phi' \mu'\bigr) = \kappa^*
\phi' K_*\mu', \quad \mbox{or, equivalently}, \quad
K_*^{\mu'}\bigl(\phi'\bigr) = \kappa^*\phi'.
\end{equation*}
\item[(2)]
Conversely, if $K_*: \Ss(\Om', \mu') \to\Ss(\Om)$ is a $\kappa
$-congruent embedding, then the following are equivalent.
\begin{enumerate}
\item[(a)]
$K_*$ is the restriction of a $\kappa$-congruent Markov morphism to
$\Ss(\Om', \mu') \subset\Ss(\Om')$.
\item[(b)]
$\mu:= K_*\mu' \in\Ss(\Om)$ admits $\kappa$-transverse measures.
\end{enumerate}
\end{enumerate}
\end{theorem}

Theorem~\ref{Markov-congruent} implies that the two notions of
congruency, that is, congruent embeddings and congruent Markov
morphisms, are equivalent for large classes of statistics $\kappa$,
since the existence of transversal measures is guaranteed under rather
mild hypotheses, for example, if one of $\Om$, $\Om'$ is a finite set,
or if $\Om$, $\Om'$ are differentiable manifolds equipped with a Borel
measure $\mu$ and $\kappa$ is a differentiable map.

However, there \emph{are} examples of statistics and measures which do
\emph{not} admit $\kappa$-transverse measures, cf. Example~\ref
{ex:notrans} below.

\begin{proof}[Proof of Theorem~\ref{Markov-congruent}]
The first statement follows directly from $(K^\kappa K)_* = (K^\kappa
)_* K_* = \kappa_* K_*$ by (\ref{eq:Markov-chainrule}) and
Proposition~\ref{prop:cong-emb}.

For the second, suppose that $K_*: \Ss(\Om', \mu') \to\Ss(\Om)$
is a $\kappa$-congruent embedding. Then $K_* = K_\mu$ given in (\ref
{congruent-mu}) for the measure $\mu:= K_*\mu'$ by Proposition~\ref
{prop:cong-emb}.

If we assume that $K_*$ is the restriction of a $\kappa$-congruent
Markov morphism induced by the $\kappa$-congruent Markov kernel $K:
\Om' \to\Pp(\Om)$, then we define the measures
\[
\mu_{{\omega}'}^\perp:= K\bigl({\omega}'
\bigr)\big|_{\kappa^{-1}({\omega})} \in \Mm\bigl(\kappa^{-1}\bigl({
\omega}'\bigr)\bigr).
\]
Note that for ${\omega}' \in\Om'$
\begin{eqnarray*}
K\bigl({\omega}'; \Om\backslash\kappa^{-1}\bigl({
\omega}'\bigr)\bigr) & = &
\int_{\Om
\backslash\kappa^{-1}({\omega}')} dK\bigl({\omega}'\bigr) =
\int_{\Om'
\backslash{\omega}'} d\kappa_*K\bigl({\omega}'\bigr)
\\
& \stackrel{\text{(\ref{eq:Cong-Mar-ker})}}= &
\int_{\Om' \backslash
{\omega}'} d\delta^{{\omega}'} = 0.
\end{eqnarray*}
That is, $K({\omega}')$ is supported on $\kappa^{-1}({\omega}')$ and
hence, for an arbitrary set $A \subset\Om$ we have
\[
K\bigl({\omega}'; A\bigr) = K\bigl({\omega}'; A \cap
\kappa^{-1}\bigl({\omega}'\bigr)\bigr) = \mu
_{{\omega}'}^\perp\bigl(A \cap\kappa^{-1}\bigl({
\omega}'\bigr)\bigr) =
\int_{\kappa
^{-1}({\omega}')} \chi_A \, d\mu_{{\omega}'}^\perp.
\]
Substituting this into the definition of $K_*$ we obtain for a subset
$A \subset\Om$
\begin{eqnarray*}
\int_\Om\chi_A \,d\mu& = & \mu(A) = K_*\bigl(
\mu'; A\bigr) \stackrel{(\ref
{eq:Markov-linear})}=
\int_{\Om'} K\bigl({\omega}';A\bigr) \,d
\mu'\bigl({\omega }'\bigr)
\\
& = &
\int_{\Om'} \biggl(
\int_{\kappa^{-1}({\omega}')} \chi_A \,d\mu_{{\omega}'}^\perp
\biggr) \,d\mu'\bigl({\omega}'\bigr),
\end{eqnarray*}
showing that (\ref{eq:mu-transverse}) holds for $\phi= \chi_A$. But
then, by linearity (\ref{eq:mu-transverse}) holds for any step
function $\phi$, and since these are dense in $L^1(\Om, \mu)$, it
follows that (\ref{eq:mu-transverse}) holds for all $\phi$, so that
the measures $\mu_{{\omega}'}^\perp$ defined above yield indeed
$\kappa$-transverse measures of $\mu$.

Conversely, suppose that $\mu:= K_*\mu'$ admits $\kappa$-transverse
measures $\mu_{{\omega}'}^\perp$, and by Proposition~\ref
{prop:transverse-probab} we may assume w.l.o.g. that $\mu_{{\omega
}'}^\perp\in\Pp(\kappa^{-1}({\omega}'))$. Then we define the map
\[
K: \Om' \longrightarrow\Pp(\Om), \qquad K\bigl({
\omega}'; A\bigr) := \mu _{{\omega}'}^\perp\bigl(A \cap
\kappa^{-1}\bigl({\omega}'\bigr)\bigr) =
\int_{\kappa
^{-1}({\omega}')} \chi_A\, d\mu_{{\omega}'}^\perp.
\]
Since for fixed $A \subset\Om$ the map ${\omega}' \mapsto\int
_{\kappa^{-1}({\omega}')} \chi_A \,  d\mu_{{\omega}'}^\perp$ is
measurable by the definition of transversal measures, $K$ is indeed a
Markov kernel. Moreover, for $A' \subset\Om'$
\[
\kappa_*K\bigl({\omega}'\bigr) \bigl(A'\bigr) = K
\bigl({\omega}'; \kappa^{-1}\bigl(A'\bigr)
\bigr) = \mu _{{\omega}'}^\perp\bigl(\kappa^{-1}
\bigl(A'\bigr) \cap\kappa^{-1}\bigl({\omega}'
\bigr)\bigr) = \delta^{{\omega}'}\bigl(A'\bigr),
\]
so that $\kappa_*K({\omega}') = \delta^{{\omega}'}$ for all
${\omega}' \in\Om'$, whence $K$ is $\kappa$-congruent. Moreover,
for any $\phi' \in L^1(\Om', \mu')$ and $A \subset\Om$ we have
\begin{eqnarray*}
K_\mu\bigl(\phi' \mu'\bigr) (A) &
\stackrel{\text{(\ref{congruent-mu})}}= & \kappa ^*\phi' \mu(A) =
\int_\Om\chi_A \kappa^*\phi' \,d
\mu
\\
& \stackrel{\text{(\ref{eq:mu-transverse})}}= &
\int_{\Om'} \biggl(
\int _{\kappa^{-1}({\omega}')} \chi_A \kappa^*\phi' \,d
\mu_{{\omega
}'}^\perp \biggr) \,d\mu'\bigl({
\omega}'\bigr)
\\
& = &
\int_{\Om'} \biggl(
\int_{\kappa^{-1}({\omega}')} \chi_A \,d\mu_{{\omega}'}^\perp
\biggr) \phi'\bigl({\omega}'\bigr) \,d
\mu'\bigl({\omega }'\bigr)
\\
& = &
\int_{\Om'} K\bigl({\omega}';A\bigr) \,d\bigl(
\phi' \mu'\bigr) \bigl({\omega}'\bigr)
\stackrel{\text{(\ref{eq:Markov-linear})}}= K_*\bigl(\phi'
\mu'\bigr) (A).
\end{eqnarray*}

Thus, $K_\mu(\phi' \mu') = K_*(\phi' \mu')$ for all $\phi' \in
L^1(\Om', \mu')$ and hence, $K_\mu(\nu) = K_*\nu$ for all $\nu\in
\Ss(\Om', \mu')$. That is, the given congruent embedding $K_\mu$
coincides with the Markov morphism $K_*$ induced by $K$, and this
completes the proof.
\end{proof}

Now we give an example of a statistic which does not admit $\kappa
$-transverse measures.

%ex3.2 #&#
\begin{example} \label{ex:notrans}
Let $\Om:= S^1$ be the unit circle group in the complex plain with the
$1$-dimensional Borel algebra ${\mathfrak B}$. Let $\Gamma:= \exp
(2\pi\sqrt{-1} {\mathbb Q}) \subset S^1$ be the subgroup of rational
rotations, and let $\Om' := S^1/\Gamma$ be the quotient space with
the canonical projection $\kappa: \Om\to\Om'$. Let ${\mathfrak B}'
:= \{A' \subset\Om'  :  \kappa^{-1}(A') \in{\mathfrak B}\}$, so
that $\kappa: \Om\to\Om'$ is measurable. For $\gamma\in\Gamma$,
we let $m_\gamma: S^1 \to S^1$ denote the multiplication by $\gamma$.

Let $\lambda$ be the $1$-dimensional Lebesgue measure on $\Om$ and
$\lambda' := \kappa_*\lambda$ be the induced measure on~$\Om'$.
Suppose that $\lambda$ admits $\kappa$-transverse measures $(\lambda
_{{\omega}'}^\perp)_{{\omega}' \in\Om'}$. Then for each $A \in
{\mathfrak B}$ we have
%
%e3.17 #&#
\begin{equation}
\label{eq:lambda} \lambda(A) =
\int_{\Om'} \biggl(
\int_{A \cap\kappa^{-1}({\omega
}')} d\lambda_{{\omega}'}^\perp \biggr)
\,d\lambda'\bigl({\omega}'\bigr).
\end{equation}
Since $\lambda$ is invariant under rotations, we have on the other
hand for $\gamma\in\Gamma$
%
%e3.18 #&#
\begin{eqnarray}
\label{eq:lambda2} \begin{aligned} \lambda(A) &= \lambda\bigl(m_\gamma^{-1}
A\bigr) =
\int_{\Om'} \biggl(
\int_{(m_\gamma^{-1} A) \cap\kappa^{-1}({\omega}')} d\lambda_{{\omega}'}^\perp \biggr)
\,d\lambda'\bigl({\omega}'\bigr)
\\
&=
\int_{\Om'} \biggl(
\int_{A \cap\kappa
^{-1}({\omega}')} d\bigl((m_\gamma)_*\lambda_{{\omega}'}^\perp
\bigr) \biggr) \,d\lambda'\bigl({\omega}'\bigr).
\end{aligned}
\end{eqnarray}
Comparing (\ref{eq:lambda}) and (\ref{eq:lambda2}) implies that
$((m_\gamma)_* \lambda_{{\omega}'}^\perp)_{{\omega}' \in\Om'}$
is another family of $\kappa$-transverse measures of $\lambda$ which
implies that $(m_\gamma)_*\lambda_{{\omega}'}^\perp= \lambda
_{{\omega}'}^\perp$ for $\lambda'$-a.e. ${\omega}' \in\Om'$, and
as $\Gamma$ is countable, it follows that
\begin{equation*}
(m_\gamma)_*\lambda_{{\omega}'}^\perp=
\lambda_{{\omega}'}^\perp \qquad \mbox{for all }\gamma\in\Gamma\mbox{
and }\lambda'\mbox{-a.e. } {\omega}' \in
\Om'.
\end{equation*}
Thus, for a.e. ${\omega}' \in\Om'$ we have $\lambda_{{\omega
}'}^\perp(\{\gamma\cdot x\}) = \lambda_{{\omega}'}^\perp(\{x\})$,
and since $\Gamma$ acts transitively on $\kappa^{-1}({\omega}')$, it
follows that singleton subsets have equal measure, i.e., there is a
constant $c_{{\omega}'}$ with
\[
\lambda_{{\omega}'}^\perp\bigl(A'\bigr) =
c_{{\omega}'} \bigl\|A'\bigr\|
\]
for all $A' \subset\kappa^{-1}({\omega}')$. As $\kappa^{-1}({\omega
}')$ is countable infinite, this implies that $\lambda_{{\omega
}'}^\perp= 0$ if $c_{{\omega}'} = 0$, and $\lambda_{{\omega
}'}^\perp(\kappa^{-1}({\omega}')) = \infty$ if $c_{{\omega}'} >
0$. Thus, $\lambda_{{\omega}'}^\perp$ is not a probability measure
for a.e. ${\omega}' \in\Om'$, contradicting Proposition~\ref
{prop:transverse-probab}. This shows that $\lambda$ does not admit
$\kappa$-transverse measures.
\end{example}

We conclude this section by the following result (cf. \cite{AJLS}, Theorem~4.10).

%th3.2 #&#
\begin{theorem} \label{thm:decomp-Markov}
Any Markov kernel $K= \Om\to\Pp(\Om')$ can be decomposed into a
statistic and a congruent Markov kernel. That is, there is a Markov
kernel $K^{\mathrm{cong}}: \Om\to\Pp(\hat\Om)$ which is congruent w.r.t.
some statistic $\kappa_1: \hat\Om\to\Om$, and a statistic $\kappa
_2: \hat\Om\to\Om'$ such that
\[
K = K^{\kappa_2} K^{\mathrm{cong}}.
\]
\end{theorem}

\begin{proof}
Let $\hat\Om:= \Om\times\Om'$ and let $\kappa_1: \hat\Om\to\Om
$ and $\kappa_2: \hat\Om\to\Om'$ be the canonical projections. We
define the Markov kernel
\[
K^{\mathrm{cong}}: \Om\longrightarrow\Pp(\hat\Om), \qquad K^{\mathrm{cong}}({
\omega }) := \delta^{\omega}\times K({\omega}),
\]
i.e., $K^{\mathrm{cong}}({\omega}; \hat A) := K({\omega}; \kappa_2(\hat A
\cap(\{{\omega}\} \times\Om')))$ for $\hat A \subset\hat\Om$.
Then evidently, $(\kappa_1)_*(K^{\mathrm{cong}}({\omega})) = \delta^{\omega
}$, so that $K^{\mathrm{cong}}$ is $\kappa_1$-congruent, and $(\kappa
_2)_*K^{\mathrm{cong}}({\omega}) = K({\omega})$, so the claim follows.
\end{proof}

%%%%%%%%%%%%%%%%%%%%%%%%%%%%%%%%%%%%
%s3.3 #&#
\subsection{Powers of densities and congruent embeddings}\label{sec3.3}
%%%%%%%%%%%%%%%%%%%%%%%%%%%%%%%%%%%%

As we saw in the preceding section, a Markov kernel $K: \Om\to\Pp(\Om')$
(e.g., a statistic $\kappa: \Om\to\Om'$), induces the monotone
bounded linear map $K_*: \Ss(\Om) \to\Ss(\Om')$ from (\ref
{eq:Markov-linear}) and for each $\mu\in\Mm(\Om)$ the restriction
yields a bounded linear map $K_*: \Ss(\Om, \mu) \to\Ss(\Om', \mu
')$, where $\mu' := K_*\mu\in\Mm(\Om')$. This induces the bounded
linear map $K_*^\mu: L^1(\Om, \mu) \to L^1(\Om', \mu')$ from (\ref
{eq:defK*mu}) (or in case of a statistic, the map $\kappa_*^\mu:
L^1(\Om, \mu) \to L^1(\Om', \mu')$ from (\ref{eq:cond-ex-L1}),
respectively).

We wish to show that when restricting this map to $L^k(\Om, \mu)
\subset L^1(\Om', \mu')$, the $k$-regularity is preserved by $\kappa
_*^\mu$ and $K_*^\mu$, respectively, cf. Theorem~\ref{thm:K-Lk}
below. The first step towards this is to consider congruent Markov kernels.

%pr3.3 #&#
\begin{proposition} \label{prop:Kr-cong}
Let $K: \Om_1 \to\Pp(\Om_2)$ be a Markov kernel which is congruent
w.r.t. some statistic $\kappa: \Om_2 \to\Om_1$. Let $\mu_1 \in\Mm
(\Om_1)$ and $\mu_2:= K_*\mu_1 \in\Mm(\Om_2)$, and consider the
map $K_*^{\mu_1}: L^1(\Om_1, \mu_1) \to L^1(\Om_2, \mu_2)$.

Then for all $\phi\in L^k(\Om_1, \mu_1)$ with $1 \leq k \leq\infty
$ we have $\phi' := K_*^\mu(\phi) \in L^k(\Om_2, \mu_2)$, and
\[
\bigl\|\phi'\bigr\|_k = \|\phi
\|_k.
\]
\end{proposition}

\begin{proof}
Since $K$ is $\kappa$-congruent, by Theorem~\ref{Markov-congruent} we
have $\phi' := K_*^{\mu_1}(\phi) = \kappa^*\phi$. Thus, for
$1 \leq k < \infty$,
%
%e3.19 #&#
\begin{equation}
\label{eq:norm-kappaphi} \bigl\|\phi'\bigr\|_k^k
=
\int_{\Om_2} \bigl\|\phi'\bigr\|
^k \,d\mu_2 =
\int_{\Om
_2}\kappa^*\|\phi\|^k \,d
\kappa_*\mu_1 =
\int_{\Om_1} \|\phi\|^k \,d
\mu_1 = \|\phi\|_k^k,
\end{equation}
showing the assertion. For $k = \infty$, $\|\phi'\|_\infty= \|\kappa
^*\phi\|_\infty= \|\phi\|_\infty$ is obvious.
\end{proof}

Next, we shall deal with statistics $\kappa: \Om\to\Om'$.

%pr3.4 #&#
\begin{proposition} \label{lem:kappa-r}
Let $\kappa: \Om\to\Om'$ be a statistic and $\mu\in\Mm(\Om)$,
$\mu' := \kappa_*\mu\in\Mm(\Om')$, and let $\kappa_*^\mu:
L^1(\Om; \mu) \to L^1(\Om', \mu')$ be the map from (\ref
{eq:cond-ex-L1}). Then the following hold.
\begin{enumerate}
\item[(1)]
If $\phi\in L^k(\Om, \mu)$ for $1 \leq k \leq\infty$, then $\phi'
:= \kappa_*^\mu(\phi) \in L^k(\Om', \mu')$, and
%
%e3.20 #&#
\begin{equation}
\label{eq:estimate-kappa} \bigl\|\phi'\bigr\|_k \leq
\|\phi\|_k.
\end{equation}
\item[(2)]
For $1 < k < \infty$, equality in (\ref{eq:estimate-kappa}) holds iff
$\phi= \kappa^*\phi'$.
\end{enumerate}
\end{proposition}

%re3.2 #&#
\begin{remark}
The estimate (\ref{eq:estimate-kappa}) in Proposition~\ref{lem:kappa-r} also follows from \cite{Neveu1965}, Proposition IV.3.1.
\end{remark}

\begin{proof}[Proof of Proposition~\ref{lem:kappa-r}]
We decompose $\phi= \phi_+ - \phi_-$ as in (\ref
{eq:phi+-}). Then $\|\phi'\|_1 = \|\kappa_*(\phi\mu)\|_{\mathrm{TV}}$ and
$\|\phi\|_1 = \|\phi\mu\|_{\mathrm{TV}}$, so that (\ref{eq:mu-abs}) implies
(\ref{eq:estimate-kappa}) for $k = 1$.

If $\phi\in L^\infty(\Om, \mu)$, then $|\phi\mu| \leq\|\phi\|
_\infty\mu$, and by monotonicity of $\kappa_*$ it follows that
\[
\bigl\|\phi' \mu'\bigr\|= \bigl\|
\kappa_*(\phi\mu)\bigr\|\leq\kappa_*\|\phi\mu\|\leq\|\phi
\|_\infty\mu',
\]
whence $\|\phi'\|_\infty\leq\|\phi\|_\infty$, so that (\ref
{eq:estimate-kappa}) holds for $k = \infty$.

For $\phi' \in L^k(\Om', \mu')$ (\ref{eq:norm-kappaphi}) implies
that $\kappa^*\phi' \in L^k(\Om, \mu)$ and
%
%e3.21 #&#
\begin{equation}
\label{eq:kappa*} \bigl\|\kappa^*\phi'\bigr\|
_k = \bigl\|\phi'\bigr\|_k
\qquad \mbox{for all }\phi' \in L^k\bigl(
\Om', \mu'\bigr), 1 \leq k \leq\infty.
\end{equation}

Suppose now that $\phi\in L^k(\Om, \mu)$ with $1 < k < \infty$ is
such that $\phi' \in L^k(\Om', \mu')$, and assume that $\phi\geq0$
and hence, $\phi' \geq0$. Then
\begin{eqnarray*}
\bigl\|\phi'\bigr\|_k^k & = &
\int_{\Om'} {\phi'}^k \,d
\mu' =
\int_{\Om'} {\phi'}^{k-1} \,d\kappa_*(
\phi\mu) =
\int_\Om\kappa^*\bigl({\phi '}^{k-1}
\bigr) \phi \,d\mu
\\
& \stackrel{(*)} \leq& \bigl\|\kappa^*\bigl({\phi'}^{k-1}
\bigr)\bigr\|_{k/(k-1)} \|\phi \|_k
\\
& \stackrel{(**)}= & \bigl\|{\phi'}^{k-1}\bigr
\|_{k/(k-1)} \|\phi\|_k = \bigl\|\phi
'\bigr\|_k^{k-1} \|\phi\|
_k.
\end{eqnarray*}
From this, (\ref{eq:estimate-kappa}) follows. Here we used H\"older's
inequality at $(*)$, and (\ref{eq:kappa*}) applied to ${\phi'}^{k-1}
\in L^{k/(k-1)}(\Om', \mu')$ at $(**)$. Moreover, equality at $(*)$
holds iff $\phi= c  \kappa^*\phi'$ for some $c \in\R$, and the
fact that $\kappa_*(\phi\mu) = \phi' \mu'$ easily implies that $c
= 1$, i.e., equality in (\ref{eq:estimate-kappa}) occurs iff $\phi=
\kappa^*\phi'$.

If we drop the assumption that $\phi\geq0$, we decompose $\phi= \phi
_+ - \phi_-$ as in (\ref{eq:phi+-}) and let $\phi'_\pm:= \kappa
_*^\mu(\phi_\pm) \geq0$. Although in general, $\phi'_+$ and $\phi
'_-$ do not have disjoint support, the linearity of $\kappa_*$ still
implies that $\phi' = \phi'_+ - \phi'_-$. Let us assume that $\phi
'_\pm\in L^k(\Om', \mu')$. Then
\[
\bigl\|\phi'\bigr\|_k = \bigl\|
\phi'_+ - \phi'_-\bigr\| _k \leq\bigl
\| \phi'_+\bigr\| _k + \bigl\| \phi
'_-\bigr\| _k \leq\| \phi_+\|
_k + \| \phi_-\| _k = \|
\phi_k\| ,
\]
using (\ref{eq:estimate-kappa}) applied to $\phi_\pm\geq0$ in the
second estimate. Equality in the second estimate holds iff $\phi_\pm=
\kappa^*\phi'_\pm$, and thus, $\phi= \phi_+ - \phi_- = \kappa
^*(\phi'_+ - \phi'_-) = \kappa^*\phi'$.

Thus, it remains to show that $\phi' \in L^k(\Om', \mu')$ whenever
$\phi\in L^k(\Om, \mu)$. For this, let $\phi\in L^k(\Om, \mu)$,
$(\phi_n)_{n \in\N}$ be a sequence in $L^\infty(\Om, \mu)$
converging to $\phi$ in $L^k(\Om, \mu)$, and let $\phi'_n := \kappa
_*^\mu(\phi_n) \in L^\infty(\Om', \mu') \subset L^k(\Om', \mu
')$. As $(\phi'_n - \phi'_m)_\pm\in L^\infty(\Om', \mu') \subset
L^k(\Om', \mu')$, (\ref{eq:estimate-kappa}) holds for $\phi_n -
\phi_m$ by the previous argument, that is,
\[
\bigl\| \phi'_n - \phi'_m
\bigr\| _k \leq\| \phi_n - \phi_m
\| _k,
\]
which tends to $0$ for $n, m \to\infty$, as $(\phi)_n$ is convergent
and hence a Cauchy sequence in $L^k(\Om, \mu)$. Thus, $(\phi')_n$ is
also a Cauchy sequence, whence it converges to some $\tilde\phi' \in
L^k(\Om', \mu')$. It follows that $\phi_n - \kappa^*\phi'_n$
converges in $L^k(\Om, \mu)$ to $\phi- \kappa^*\tilde\phi'$, and
as $\kappa_*((\phi_n - \kappa^*\phi'_n)\mu) = 0$ for all $n$, we have
\[
0 = \kappa_*\bigl(\bigl(\phi- \kappa^*\tilde\phi'\bigr)\mu\bigr) =
\phi' \mu' - \tilde\phi'
\mu',
\]
whence $\phi' = \tilde\phi' \in L^k(\Om', \mu')$.
\end{proof}

Putting the last two results together, we obtain the following
theorem.

%th3.3 #&#
\begin{theorem} \label{thm:K-Lk}
Let $K: \Om\to\Pp(\Om')$ be a Markov kernel, $\mu\in\Mm(\Om)$
and $\mu' := K_*\mu\in\Mm(\Om')$. Then for $\phi\in L^k(\Om, \mu
)$ with $1 \leq k \leq\infty$ we have $K_*^\mu(\phi) \in L^k(\Om',
\mu')$, and
\[
\bigl\| K_*^\mu(\phi)\bigr\| _k \leq\|
\phi\| _k.
\]
\end{theorem}

\begin{proof}
By Theorem~\ref{thm:decomp-Markov}, we can decompose $K
= \kappa_* K^{\mathrm{cong}}$, where $K^{\mathrm{cong}}: \Om\to\Pp(\hat\Om)$ is
congruent w.r.t. some statistic $\hat\kappa: \hat\Om\to\Om$, and
with a statistic $\kappa: \hat\Om\to\Om'$. Then it follows that
$K_* = \kappa_* K^{\mathrm{cong}}_* $, and whence,
\[
K_*^\mu= \kappa_*^{\hat\mu} {K^{\mathrm{cong}}_*}^\mu,
\]
where $\hat\mu:= K^{\mathrm{cong}}_*(\mu) \in\Mm(\hat\Om)$. Given $\phi
\in L^k(\Om, \mu)$, then by Theorem~\ref{Markov-congruent}, $\hat
\phi:= {K^{\mathrm{cong}}_*}^{\mu}(\phi) = \hat\kappa^*\phi$, whence $\phi
' := K_*^\mu(\phi) = \kappa_*^{\hat\mu}(\hat\phi)$. Thus,
\[
\bigl\| K_*^\mu(\phi)\bigr\| _k = \bigl\|
\kappa_*^{\hat\mu}(\hat\phi)\bigr\| _k \leq\| \hat\phi
\| _k = \| \phi\| _k,
\]
where the first estimate follows from Proposition~\ref{lem:kappa-r},
whereas the second equation follows from Proposition~\ref{prop:Kr-cong}.
\end{proof}

%re3.3 #&#
\begin{remark}\label{rem3.3}
Theorem~\ref{thm:K-Lk} can be interpreted in a different way. Namely,
given a Markov kernel $K: \Om\to\Pp(\Om')$ and $r \in(0,1]$, one
can define the map $K_*^r: \Ss^r(\Om) \to\Ss^r(\Om')$ by
%
%e3.22 #&#
\begin{equation}
\label{eq:def-Kr} K^r_*\bigl(\tilde\mu^r\bigr) := \tilde
\pi^r(K_*\mu) \qquad \mbox{for }\mu \in\Ss(\Om),
\end{equation}
with the signed $r$th power $\tilde\mu^r$ defined before. Since
$\tilde\pi^r$ and $\tilde\pi^{1/r}$ are both continuous by
Proposition~\ref{prop:powers-C1}, the map $K^r_*$ is continuous, but
it fails to be $C^1$ for $r < 1$, even for finite $\Om$.

Let $\mu\in\Mm(\Om)$ and $\mu' := K_*\mu\in\Mm(\Om')$, so that
$K_*^r(\mu^r) = {\mu'}^r$. If there was a derivative of $K_*^r$ at
$\mu^r$, then it would have to be a map between the tangent spaces
$T_{\mu^r}\Mm(\Om)$ and $T_{{\mu'}^r}\Mm(\Om')$, i.e., according
to Proposition~\ref{prop:TMr-SMr} between $\Ss^r(\Om, \mu)$ and
$\Ss^r(\Om', \mu')$. Let $k := 1/r > 1$, $\phi\in L^k(\Om, \mu)
\subset L^1(\Om, \mu)$, so that $\phi' := K_*^\mu(\phi) \in
L^k(\Om', \mu')$ by Theorem~\ref{thm:K-Lk}. Then by Proposition~\ref
{prop:powers-C1} and the chain rule we obtain
\begin{eqnarray*}
d\bigl(\tilde\pi^k K_*^r\bigr)_{\mu^r}\bigl(\phi
\mu^r\bigr) & = & k {\mu'}^{1-r} \cdot d
\bigl(K_*^r\bigr)_{\mu^r}\bigl(\phi\mu^r\bigr),
\\
d\bigl(K_* \tilde\pi^k\bigr)_{\mu^r}\bigl(\phi
\mu^r\bigr) & = & k K_*(\phi\mu) = k \phi'
\mu',
\end{eqnarray*}
and these should coincide as $\tilde\pi^k K_*^r = K_* \tilde\pi^k$
by (\ref{eq:def-Kr}). Since $d(K_*^r)_{\mu^r}(\phi\mu^r) \in\Ss
^r(\Om', \mu')$, we thus must have
%
%e3.23 #&#
\begin{equation}
\label{def-dKr} d\bigl(K_*^r\bigr)_{\mu^r}\bigl(\phi
\mu^r\bigr) = \phi' {\mu'}^r,
\qquad \mbox{where } \phi' = K_*^\mu(\phi).
\end{equation}

Thus, Theorem~\ref{thm:K-Lk} states that this map is a well defined
linear operator with operator norm $\leq1$. The map $d(K_*^r)_{\mu
^r}: \Ss^r(\Om, \mu) \to\Ss^r(\Om', \mu')$ from (\ref{def-dKr})
is called the \emph{formal derivative of} $K_*^r$ \emph{at} $\mu^r$.
\end{remark}

%%%%%%%%%%%%%%%%%%%%%%%%%%%%%%%%%%%%
%s4 #&#
\section{Parametrized measure models and $k$-integrability} \label{sec:k-int}
%%%%%%%%%%%%%%%%%%%%%%%%%%%%%%%%%%%%

In this section, we shall now present our notion of a parametrized
measure model.

%de4.1 #&#
\begin{definition}[Parametrized measure model] \label{def:param-measure-model}
Let $\Om$ be a measure space.
\begin{enumerate}
\item[(1)]
A \emph{parametrized measure model} is a triple $(M, \Om, \mathbf{p})$
where $M$ is a (finite or infinite dimensional) Banach manifold and
$\mathbf{p}: M \to\Mm(\Om) \subset\Ss(\Om)$ is a $C^1$-map in the
sense explained in Section~\ref{sec:Banach}.

\item[(2)]
The triple $(M, \Om, \mathbf{p})$ is called a \emph{statistical model} if
it consists only of probability measures, that is, such that the image
of $\mathbf{p}$ is contained in $\Pp(\Om)$.

\item[(3)]
We call such a model \emph{dominated by} $\mu_0$ if the image of $\mathbf{p}$ is contained in $\Ss(\Om, \mu_0)$. In this case, we use the
notation $(M, \Om, \mu_0, \mathbf{p})$ for this model.
\end{enumerate}
\end{definition}

%re4.1 #&#
\begin{remark}\label{rem4.1}
Evidently, for the applications we have in mind, we are interested
mainly in statistical models. However, we can take the point of view
that $\Pp(\Om)$ is the projectivization of $\Pp(\Om) = {\mathbb
P}(\Mm(\Om) \backslash0)$ via rescaling. Thus, given a parametrized
measure model $(M, \Om, \mathbf{p})$, normalization yields a statistical
model $(M, \Om, \mathbf{p}_0)$ defined by
\[
\mathbf{p}_0(\xi) := \frac{\mathbf{p}(\xi)}{\| \mathbf{p}(\xi) \|_{\mathrm{TV}}},
\]
which is again a $C^1$-map. Indeed, the map $\mu\mapsto\|\mu\|_{\mathrm{TV}}$
on $\Mm(\Om)$ is a $C^1$-map, being the restriction of the linear
(and hence differentiable) map $\mu\mapsto\int_\Om d\mu$ on $\Ss
(\Om)$.

Observe that while $\Ss(\Om)$ is a Banach space, the subsets $\Mm
(\Om)$ and $\Pp(\Om)$ do not carry a canonical manifold structure.
\end{remark}

If a parametrized measure model $(M, \Om, \mu_0, \mathbf{p})$ is
dominated by $\mu_0$, then there is a \emph{density function} $p: \Om
\times M \to\R$ such that
%
%e4.1 #&#
\begin{equation}
\label{eq:density-fctn} \mathbf{p}(\xi) = p(\cdot;\xi) \mu_0.
\end{equation}

Evidently, we must have $p(\cdot;\xi) \in L^1(\Om, \mu_0)$ for all $\xi
$. In particular, for fixed $\xi$, $p(\cdot;\xi)$ is defined only up to
changes on a $\mu_0$-null set.

%de4.2 #&#
\begin{definition}[Regular density function] \label{def:regular-density}

Let $(M, \Om, \mu_0, \mathbf{p})$ be a parametrized measure model
dominated by $\mu_0$. We say that this model has a \emph{regular
density function} if the density function $p: \Om\times M \to\R$
satisfying (\ref{eq:density-fctn}) can be chosen such that for all $V
\in T_\xi M$ the partial derivative $\partial_V p(\cdot;\xi)$ exists and
lies in $L^1(\Om, \mu_0)$.
\end{definition}

%re4.2 #&#
\begin{remark} \label{rem:paramMeas}
The standard notion of a statistical model always assumes that it is
dominated by some measure and has a positive regular density function
(e.g., \cite{Amari1987}, Section~2, p.~25, \cite{AN2000}, Section~2.1, \cite
{Rao1945}, \cite{AJLS}, Definition~2.4). In fact, the definition of a
parametrized measure model or statistical model in \cite{AJLS}, Definition~2.4, is equivalent to a parametrized measure model or statistical
model with a positive regular density function in the sense of
Definition~\ref{def:regular-density}.

Let us point out why the present notion is indeed more general.
The formal definition of differentiability of $\mathbf{p}$ implies that
for each $C^1$-path $\xi(t) \in M$ with $\xi(0) = \xi$, $\dot\xi
(0) =: V \in T_\xi M$, the curve $t \mapsto p(\cdot;\xi(t)) \in L^1(\Om,
\mu_0)$ is differentiable. This implies that there is a $d_\xi\mathbf{p}(V) \in L^1(\Om, \mu_0)$ such that
\[
\biggl\| \frac{p(\cdot;\xi(t)) - p(\cdot;\xi)}t - d_\xi\mathbf{p}(V) (\cdot) \biggr
\| _1 \xrightarrow{t \to0} 0.
\]
If this is a \emph{pointwise} convergence, then $d_\xi\mathbf{p}(V) =
\partial_V p(\cdot; \xi)$ is the partial derivative and whence, $\partial
_V p(\cdot; \xi)$ lies in $L^1(\Om, \mu_0)$, so that the density
function is regular.

However, in general convergence in $L^1(\Om, \mu_0)$ does \emph{not}
imply pointwise convergence, whence there are parametrized measure
models in the sense of Definition~\ref{def:param-measure-model}
without a regular density function, cf. Example~\ref{ex:param}
below. Nevertheless, for simplicity we shall frequently use the
notation $\partial_V p(\cdot; \xi)$ instead of $d_\xi\mathbf{p}(V)(\cdot)$, even if the density function is \emph{not} regular.
\end{remark}

By this convention, for a parametrized measure model $(M, \Om, \mu_0,
\mathbf{p})$ we can describe its derivative in the direction of $V \in
T_\xi M$ as
\begin{equation*}
d_\xi\mathbf{p}(V) = \partial_V p(\cdot;\xi)
\mu_0.
\end{equation*}

%ex4.1 #&#
\begin{example} \label{ex:param}
To see that there are parametrized measure models without a regular
density function, consider the family of measures on $\Om= (0,\pi)$
and $\xi\in(-1, \infty)$
\[
\mathbf{p}(\xi) := p(t; \xi) \, dt \qquad \mbox{with } p(t; \xi) = %
\begin{cases} \bigl(1 + \xi \bigl(\sin^2 (t - 1/\xi)
\bigr)^{1/\xi^2} \bigr) \, dt & \mbox{for }\xi\neq0,
\\
1 & \mbox{for }\xi= 0. \end{cases} %
\]
This model is dominated by the Lebesgue measure $dt$ with density
function $p$, and the partial derivative $\partial_\xi p$ does not
exist at $\xi= 0$, whence the density function is not regular.

On the other hand, $\mathbf{p}: \R\to\Mm(\Om, dt)$ is differentiable
in the above sense at $\xi= 0$ with $d_0\mathbf{p}(\partial_\xi) = 0$,
so that $(M, \Om, \mathbf{p})$ is a parametrized measure model in the
sense of Definition~\ref{def:param-measure-model}. To see this, we calculate
\begin{eqnarray*}
\biggl\| \frac{\mathbf{p}(\xi) - \mathbf{p}(0)}\xi\biggr\| _1 & = & \bigl
\| \bigl(\sin^2 (t - 1/\xi)\bigr)^{1/\xi^2} \,dt \bigr
\| _1
\\
& = &
\int_0^\pi\bigl(\sin^2 (t - 1/\xi)
\bigr)^{1/\xi^2} \,dt
\\
& = &
\int_0^\pi\bigl(\sin^2 t
\bigr)^{1/\xi^2} \,dt \xrightarrow{\xi\to0} 0,
\end{eqnarray*}
which shows the claim. Here, we used the $\pi$-periodicity of the
integrand for fixed $\xi$ and dominated convergence in the last step.
\end{example}

Since for a parametrized measure model $(M, \Om, \mathbf{p})$ the map
$\mathbf{p}$ is $C^1$, it follows that its derivative yields a continuous
map between the tangent fibrations
\[
d\mathbf{p}: TM \longrightarrow T\Mm(\Om) = \mathop{\dot\bigcup
}_{\mu\in\Mm
(\Om)} \Ss(\Om, \mu).
\]
That is, for each tangent vector $V \in T_\xi M$, its differential
$d_\xi\mathbf{p}(V)$ is contained in $\Ss(\Om, \mathbf{p}(\xi))$ and
hence dominated by $\mathbf{p}(\xi)$.

%de4.3 #&#
\begin{definition}\label{def:logarithmic}
Let $(M, \Om, \mathbf{p})$ be a parametrized measure model. Then for each
tangent vector $V \in T_\xi M$ of $M$, we define
%
%e4.2 #&#
\begin{equation}
\label{eq:deriv-log} \partial_V \log\mathbf{p}(\xi) := \frac{d\{d_\xi\mathbf{p}(V)\}}{d\mathbf{p}(\xi)}
\in L^1\bigl(\Om, \mathbf{p}(\xi)\bigr)
\end{equation}
and call this the \emph{logarithmic derivative of} $\mathbf{p}$ \emph{at} $\xi$ \emph{in
direction} $V$.
\end{definition}

If such a model is dominated by $\mu_0$ and has a regular density
function $p$ for which (\ref{eq:density-fctn}) holds, then we can
calculate the Radon--Nikodym derivative as
\begin{eqnarray*}
\frac{d\{d_\xi\mathbf{p}(V)\}}{d\mathbf{p}(\xi)} & = & \frac{d\{d_\xi
\mathbf{p}(V)\}}{d\mu_0} \cdot \biggl(\frac{d\mathbf{p}(\xi)}{d\mu
_0}
\biggr)^{-1}
\\
& = & \partial_V p(.;\xi) \bigl(p(.;\xi)\bigr)^{-1} =
\partial_V \log p(\cdot;\xi),
\end{eqnarray*}
where we use the convention $\log0 = 0$. This justifies the notation
in (\ref{eq:deriv-log}) even for models without a regular density function.

For a parametrized measure model $(M, \Om, \mathbf{p})$ and $k > 1$ we
consider the map
\begin{equation*}
\mathbf{p}^{1/k} := \pi^{1/k} \circ\mathbf{p}: M
\longrightarrow\Ss ^{1/k}(\Om), \qquad \xi\longmapsto\mathbf{p}(
\xi)^{1/k}.
\end{equation*}

Since $\pi^{1/k}$ is continuous by Proposition~\ref{prop:powers-C1},
it follows that $\mathbf{p}^{1/k}$ is continuous as well. Let us pretend
for the moment that $\mathbf{p}^{1/k}$ is a $C^1$-map, so that $d_\xi\mathbf{p}^{1/k}(V) \in T_{\mathbf{p}(\xi)^{1/k}}\Mm^{1/k}(\Om) = \Ss
^{1/k}(\Om, \mathbf{p}(\xi))$. In this case, because of $\pi^k \circ
\pi^{1/k} = \mathrm{Id}$, we have
\[
\mathbf{p} = \pi^k \circ\mathbf{p}^{1/k},
\]
whence by the chain rule and (\ref{eq:form-dpi}) we have for $\xi\in
M$ and $V \in T_\xi M$
\begin{equation*}
d_\xi\mathbf{p}(V) = k \mathbf{p}(\xi)^{1-1/k} \cdot
\bigl(d_\xi\mathbf{p}^{1/k}(V) \bigr).
\end{equation*}
Thus with (\ref{eq:deriv-log}) this implies
%
%e4.3 #&#
\begin{equation}
\label{eq:formal-derivative} d_\xi\mathbf{p}^{1/k}(V) =
\frac{1}k \partial_V \log\mathbf{p}(\xi)
\mathbf{p}^{1/k}(\xi) \in\Ss^{1/k}\bigl(\Om, \mathbf{p}(\xi)
\bigr)
\end{equation}
and hence, in particular, $\partial_V \log\mathbf{p}(\xi) \in L^k(\Om,
\mathbf{p}(\xi))$, and depends continuously on $V \in TM$. This motivates
the following definition.

%de4.4 #&#
\begin{definition}[$k$-integrable
parametrized measure model]\label{def:k-integrable}

A parametrized measure model $(M, \Om, \mathbf{p})$ is called $k$-\emph{integrable} for $k \geq1$ if for all $\xi\in M$ and $V \in T_\xi
M$ we have
\[
\partial_V \log\mathbf{p}(\xi) = \frac{d\{d_\xi\mathbf{p}(V)\}}{d\mathbf{p}(\xi)} \in
L^k\bigl(\Om, \mathbf{p}(\xi)\bigr),
\]
and moreover, the map
\[
d \mathbf{p}^{1/k}: T M \longrightarrow T\Ss^{1/k}(\Om)
\]
given in (\ref{eq:formal-derivative}) is continuous. Furthermore,
we call the model $\infty$-\emph{integrable} if it is $k$-integrable
for all $k \geq1$.
\end{definition}

Since $\mathbf{p}(\xi)$ is a finite measure, we have $L^k(\Om, \mathbf{p}(\xi)) \subset L^l(\Om, \mathbf{p}(\xi))$ for all $1 \leq l \leq k$.
Thus, $k$-integrability implies $l$-integrability for all such $l$.

%re4.3 #&#
\begin{remark} \label{rem:k-integ}
The reader who is familiar with the definition of $k$-integrability in \cite{AJLS}, Definition~2.4. might notice that there only the continuity of the norm $\|d_\xi \mathbf{p}^{1/k}\|$ on $TM$ is required. Also, by our previous discussion, a parametrized measure model $(M, \Om,
\mathbf{p})$ for which $\mathbf{p}^{1/k}$ is a $C^1$-map is always
$k$-integrable.

As it turns out, these three notions of $k$-integrability are equivalent. That is, $\|d_\xi \mathbf{p}^{1/k}\|$ is continuous if and only if $d_\xi \mathbf{p}^{1/k}$ is continuous if and only if $\mathbf{p}^{1/k}$ is a $C^1$-map.
\end{remark}

%de4.5 #&#
\begin{definition}[Canonical $n$-tensor] \label{def:canonical-n-tensor}
For $n \in\N$, the \emph{canonical} $n$-\emph{tensor} is the covariant
$n$-tensor on $\Ss^{1/n}(\Om)$, given by
%
%e4.4 #&#
\begin{equation}
\label{def:canonical-tensor} L^n_\Om(\nu_1, \ldots,
\nu_n) = n^n
\int_\Om d(\nu_1 \cdots\nu _n),
\qquad \mbox{where }\nu_i \in\Ss^{1/n}(\Om).
\end{equation}
\end{definition}

The main purpose of defining the notion of $k$-integrability is that
for a $k$-integrable model, we can for any $n \leq k$ define the pullback
%
%e4.5 #&#
\begin{eqnarray}
\label{eq:define-tau-n-M} \begin{aligned} \tau^n_{(M, \Om, \mathbf{p})} & := \bigl(\mathbf{p}^{1/n}\bigr)^\ast L^n_\Om, \qquad \qquad \text{whence}\\
\tau^n_{(M, \Om, \mathbf{p})}(V_1,
\ldots, V_n) & = L^n_\Om
\bigl(d_\xi \mathbf{p}^{1/n}(V_1), \ldots,
d_\xi\mathbf{p}^{1/n}(V_n)\bigr)
\\
& =
\int_\Om\partial_{V_1} \log\mathbf{p}(\xi) \cdots\, 
\partial _{V_n} \log\mathbf{p}(\xi) \, d\mathbf{p}(\xi), \end{aligned}
\end{eqnarray}
where the second line follows immediately from (\ref
{eq:formal-derivative}) and (\ref{def:canonical-tensor}). 
This is well defined as $\mathbf{p}^{1/n}: M \to{\cal S}^{1/n}(\Om)$ is differentiable by Remark~\ref{rem:k-integ}

%ex4.2 #&#
\begin{example}\label{ex4.2}
\begin{enumerate}
\item[(1)] For $n=1$, the canonical $1$-form is given as
\begin{equation*}
\tau^1_{(M, \Om, \mathbf{p})}(V) :=
\int_\Om\p_V \log\mathbf{p}(\xi) \, d\mathbf{p}(
\xi) = \p_V \bigl\| \mathbf{p}(\xi)\bigr\| .
\end{equation*}
Thus, it vanishes if and only if $\|\mathbf{p}(\xi)\|$ is locally
constant, e.g., if $(M, \Om, \mathbf{p})$ is a \emph{statistical} model.
\item[(2)] For $n = 2$, $\tau^2_{(M, \Om, \mathbf{p})}$ coincides with the
\textit{Fisher metric}
%
%e4.6 #&#
\begin{equation}
\g^F(V, W)_\xi : =
\int_\Om\p_V \log\mathbf{p}(\xi) \,
\p_W \log \mathbf{p}(\xi) \, d\mathbf{p}(\xi)\label{for:fisher}
\end{equation}
\item[(3)] For $n=3$, $\tau^3_{(M, \Om, \mathbf{p})}$ coincides with the
\textit{Amari--Chentsov} 3-\textit{symmetric tensor}
\begin{equation*}
T^{\mathrm{AC}} (V,W, X)_\xi: =
\int_\Om\p_V \log\mathbf{p}(\xi)
\, \p_W \log\mathbf{p}(\xi) \, \p_X \log\mathbf{p}(\xi) \,
d \mathbf{p}(\xi).
\end{equation*}
\end{enumerate}
\end{example}

Observe that the Fisher metric $\g^F$ is a Riemannian metric on $M$
iff $\mathbf{p}$ is an immersion, i.e., if $\ker d_\xi\mathbf{p} = 0$.

%re4.4 #&#
\begin{remark}\label{rem4.4}
While the Fisher metric and the Amari--Chentsov tensor give an
interpretation of $\tau_{(M, \Om, \mathbf{p})}^n$ for $n = 2,3$, we do
not know of any statistical significance of $\tau_{(M, \Om, \mathbf{p})}^n$ for $n \geq4$. However, we shall show later that $\tau
^{2n}_M$ can be used to measure the information loss of statistics and
Markov kernels, cf. Theorem~\ref{thm:monotonicity}. Moreover, in \cite{Lauritzen1987}, p.~212, the question is posed if there are other
significant tensors on statistical manifolds, and the canonical
$n$-tensors may be considered as natural candidates.
\end{remark}

%%%%%%%%%%%%%%%%%%%%%%%%
%s5 #&#
\section{Parametrized measure models and sufficient statistics} \label{sec:suffstat}
%%%%%%%%%%%%%%%%%%%%%%%%

Given a parametrized measure model (statistical model, respectively)
$(M, \Om, \mathbf{p})$ and a Markov kernel $K: \Om\to\Pp(\Om')$ which
induces the Markov morphism $K_*: \Mm(\Om) \to\Mm(\Om')$ as in
(\ref{eq:Markov-linear}), we obtain another parametrized measure model
(statistical model, respectively) $(M, \Om', \mathbf{p}')$ by defining
$\mathbf{p}'(\xi) := K_*\mathbf{p}(\xi)$. These transitions can be
interpreted as data processing in statistical decision theory, which
can be deterministic (i.e., given by a statistic) or randomized (i.e.,
given by a Markov kernel). We refer to for example, \cite
{Chentsov1982} where this is elaborated in detail.

It is the purpose of this section to investigate the relation between
these two models in more detail.

%th5.1 #&#
\begin{theorem}\label{thm:induced-kintegrable}
Let $(M, \Om, \mathbf{p})$, $K: \Om\to\Pp(\Om')$ and $(M, \Om', \mathbf{p}')$ be as above, and suppose that $(M, \Om, \mathbf{p})$ is
$k$-integrable for some $k \geq1$. Then $(M, \Om', \mathbf{p}')$ is also
$k$-integrable, and
%
%e5.1 #&#
\begin{equation}
\label{eq:est-log} \bigl\| \p_V \log\mathbf{p}'(\xi)
\bigr\| _k \leq\bigl\| \p_V \log\mathbf{p}(\xi)
\bigr\| _k\qquad \mbox{for all }V \in T_\xi M,
\end{equation}
where the norms are taken in $L^k(\Om, \mathbf{p}(\xi))$ and $L^k(\Om',
\mathbf{p}'(\xi))$, respectively. If $K$ is congruent, then equality in
(\ref{eq:est-log}) holds for all $V$.

Moreover, if $K$ is given by a statistic $\kappa: \Om\to\Om'$ and
$k > 1$, then equality in (\ref{eq:est-log}) holds iff $\p_V \log
\mathbf{p}(\xi) = \kappa^*(\p_V \log\mathbf{p}'(\xi))$. In particular,
equality in (\ref{eq:est-log}) either holds for \emph{all} $k > 1$ for
which the model is $k$-integrable or for \emph{no} such $k > 1$.
\end{theorem}

\begin{proof}
Since $K_*$ is the restriction of a bounded linear map, it is obvious
that $\mathbf{p}': M \to\Mm(\Om')$ is again differentiable, and in fact,
%
%e5.2 #&#
\begin{equation}
\label{eq:p-p'} d_\xi\mathbf{p}'(V) = K_*
\bigl(d_\xi\mathbf{p}(V)\bigr)
\end{equation}
for all $V \in T_\xi M$, $\xi\in M$.

Let $\mu:= \mathbf{p}(\xi)$ and $\mu' := \mathbf{p}'(\xi) = K_*\mu$, and
let $\phi:= \p_V \log\mathbf{p}(\xi)$ and $\phi' :=\p_V \log\mathbf{p}'(\xi)$, so that
$d_\xi\mathbf{p}(V) = \phi\mu$ and $d_\xi\mathbf{p}'(V) = \phi' \mu'$. By (\ref{eq:p-p'}) we thus have
\[
K_*(\phi\mu) = \phi' \mu',
\]
so that $\phi' = K_*^\mu(\phi)$ is the expectation value of $\phi$
given $K$. If $\mathbf{p}$ is $k$-integrable, then
$\phi= \p_V \log\mathbf{p}(\xi) \in L^k(\Om, \mu)$, whence $\phi' \in L^k(\Om', \mu')$,
and $\|\phi'\|_k \leq\|\phi\|_k$, by Theorem~\ref{thm:K-Lk}. That
is, $\mathbf{p}'$ is $k$-integrable as well and (\ref{eq:est-log}) holds.

If $K$ is congruent, then $\|\phi'\|_k = \|\phi\|_k$ by Proposition~\ref{prop:Kr-cong}.

If $k > 1$ and $K$ is given by a statistic $\kappa$, then equality in
(\ref{eq:est-log}) occurs iff $\phi= \kappa^*\phi'$ by Proposition~\ref{lem:kappa-r}.
\end{proof}

Since the Fisher metrics $\g^F$ of $(M, \Om, \mathbf{p})$ and ${\g'}^F$
of $(M, \Om', \mathbf{p}')$ are defined as
\[
\g(V,V) = \bigl\| \p_V \log\mathbf{p}(\xi)\bigr\|
_2^2 \quad \mbox{and}\quad \g'(V, V) =
\bigl\| \p_V \log\mathbf{p}'(\xi)\bigr\|
_2^2
\]
by (\ref{for:fisher}), Theorem~\ref{thm:induced-kintegrable}
immediately implies the following.

%th5.2 #&#
\begin{theorem}[Monotonicity theorem, cf.
\cite{AN2000,AJLS,AJLS2}]\label{thm:monotonicity}
Let $(M, \Om, \mathbf{p})$ be a $k$-integrable parame\-trized measure model
for $k \geq2$, let $K: \Om\to\Pp(\Om')$ be a Markov kernel, and
let $(M, \Om', \mathbf{p}')$ be given by $\mathbf{p}'(\xi) = K_*\mathbf{p}(\xi
)$. Then
%
%e5.3 #&#
\begin{equation}
\label{eq:monotonicity} \g(V, V) \geq\g'(V, V)\qquad \mbox{for all }V \in
T_\xi M\mbox{ and } \xi\in M.
\end{equation}
\end{theorem}

%re5.1 #&#
\begin{remark}\label{rem5.1}
Note that our approach allows to prove the Monotonicity Theorem~\ref
{thm:monotonicity} with no further assumption on the model $(M, \Om,
\mathbf{p})$. In order for (\ref{eq:monotonicity}) to hold we can work
with arbitrary Markov kernels, not just statistics $\kappa$. Even if
$K$ is given by a statistic $\kappa$, we do not need to assume that
$\Om$ is a topological space with its Borel $\sigma$-algebra as in
\cite{Le2013}, Theorem~1.2, nor do we need to assume the existence of
transversal measures of the map $\kappa$ (e.g. \cite{AN2000}, Theorem~2.1), nor do we need to assume that all measures $\mathbf{p}(\xi
)$ have the same null sets (\cite{AJLS}, Theorem~3.11). In this sense,
our statement generalizes these versions of the monotonicity theorem,
as it even covers a rather peculiar statistic as in Example~\ref{ex:notrans}.
\end{remark}

In \cite{AN2000}, p.~98, the difference
%
%e5.4 #&#
\begin{equation}
\label{eq:InfoLoss-classical} \g(V, V) - \g'(V, V) \geq0
\end{equation}
is called the \emph{information loss} of the model under the statistic
$\kappa$, a notion which is highly relevant for statistical inference.
This motivates the following definition.

%de5.1 #&#
\begin{definition}\label{def:infoloss}
Let $(M, \Om, \mathbf{p})$ be $k$-integrable for some $k > 1$, let $K:
\Om\to\Pp(\Om')$ and $(M, \Om', \mathbf{p}')$ be as above, so that
$(M, \Om', \mathbf{p}')$ is $k$-integrable as well. Then for each $V \in
T_\xi M$ we define the $k$\emph{th order information loss under} $K$ \emph{in
direction} $V$ as
\[
\bigl\| \p_V \log\mathbf{p}(\xi)\bigr\|
_k^k - \bigl\| \p_V \log
\mathbf{p}'(\xi)\bigr\| _k^k \geq0,
\]
where the norms are taken in $L^k(\Om, \mathbf{p}(\xi))$ and $L^k(\Om',
\mathbf{p}'(\xi))$, respectively.
\end{definition}

That is, the information loss in (\ref{eq:InfoLoss-classical}) is
simply the special case $k = 2$ in Definition~\ref{def:infoloss}.
Observe that due to Theorem~\ref{thm:induced-kintegrable} the
vanishing of the information loss for \emph{some} $k>1$ implies the
vanishing for \emph{all} $k > 1$ for which this norm is defined. That
is, the $k$th order information loss measures the same quantity by
different means.

For instance, if $(M, \Om, \mathbf{p})$ is $k$-integrable for $1 < k <
2$, but not $2$-integrable, then the Fisher metric and hence the
classical information loss in (\ref{eq:InfoLoss-classical}) is not
defined. Nevertheless, we still can quantify the $k$th order
information loss of a statistic of this model.

Observe that for $k=2n$ an even integer $\tau^{2n}_{(M, \Om, \mathbf{p})}(V, \ldots, V) = \|\p_V \log\mathbf{p}(\xi)\|_{2n}^{2n}$, whence
the difference
\[
\tau^{2n}_{(M, \Om, \mathbf{p})}(V, \ldots, V)- \tau^{2n}_{(M, \Om',
\mathbf{p}')}(V,
\ldots, V) \geq0
\]
represents the $2n$th order information loss of $\kappa$ in direction
$V$. This gives an interpretation of the canonical $2n$-tensors $\tau
^{2n}_{(M, \Om, \mathbf{p})}$.

It is a natural problem to characterize statistics of a model which do
not produce any information loss. Fisher \cite{Fisher1922} called such
a statistic \emph{sufficient} writing that ``\ldots the criterion of
sufficiency, which latter requires that the whole of the relevant
information supplied by a sample shall be contained in the statistics
calculated'' \cite{Fisher1922}, p.~367. This motivates the following
definition.

%de5.2 #&#
\begin{definition}[Sufficient statistic]\label{def:suffstat2}
Let $(M, \Om, \mathbf{p})$ be a parametrized measure model which is
$k$-integrable for some $k > 1$. Then a statistic $\kappa: \Om\to\Om
'$ or, more general, a Markov kernel $K: \Om\to{\cal P}(\Om')$ is
called a \emph{sufficient for the model} if the $k$th order
information loss vanishes for all tangent vectors $V$, that is, if
\[
\bigl\| \p_V \log\mathbf{p}'(\xi)\bigr\|
_k = \bigl\| \p_V \log\mathbf{p}(\xi)\bigr\|
_k\qquad \mbox{for all }V \in T_\xi M,
\]
where $\mathbf{p}'(\xi) = \kappa_*\mathbf{p}(\xi)$ or $\mathbf{p}'(\xi) =
K_*\mathbf{p}(\xi)$, respectively.
\end{definition}

Again, in this definition it is irrelevant which $k > 1$ is used, as
long as $k$-integrability of the model is satisfied.

%ex5.1 #&#
\begin{example}[Fisher--Neyman \cite{Neyman1935}] \label{ex:sufficient}
Let $(M, \Om', \mu_0', \mathbf{p}')$ be a parametrized measure model
dominated by $\mu_0'$, given by
\[
\mathbf{p}'(\xi) = \phi'(\cdot;\xi)
\mu_0', \qquad \phi': \Om\times M
\longrightarrow[0,\infty].
\]
Moreover, let $\kappa: \Om\to\Om'$ be a statistic and $\mu_0 \in
\Mm(\Om)$ such that $\kappa_*(\mu_0) = \mu_0'$. Define the
parametrized measure model $(M, \Om, \mu_0, \mathbf{p})$ as
%
%e5.5 #&#
\begin{equation}
\label{Fisher-Neyman} \mathbf{p}(\xi) := \phi'\bigl(\kappa(\cdot), \xi
\bigr) \mu_0.
\end{equation}
Then $\kappa$ is a sufficient statistic for $(M, \Om, \mu_0, \mathbf{p})$.
Indeed, $\kappa_*(\mathbf{p})(\xi) = \mathbf{p}'(\xi)$ by (\ref
{eq:pushforward-mult}), and $d_\xi\mathbf{p}(V) = \kappa^*(d\mathbf{p}'_\xi
(V))$ for all $V \in T_\xi M$, so that $\p_V \log\mathbf{p}(\xi) =
\kappa^*(\p_V \log\mathbf{p}'(V))$. By Theorem~\ref
{thm:induced-kintegrable} it follows that equality holds in (\ref
{eq:est-log}), so that $\kappa$ is a sufficient statistic for $(M, \Om
, \mu_0, \mathbf{p})$.
\end{example}

Under some further assumptions, the statistics given in Example~\ref
{ex:sufficient} exhaust \emph{all} sufficient statistics. More
precisely, the following is known as the \emph{Fisher--Neyman factorization}.

%pr5.1 #&#
\begin{proposition}[\cite{Neyman1935}]\label{prop:infoloss}
Let $(M, \Om, \mu, \mathbf{p})$ be a parametrized measure model with a
positive regular density function $p: \Om\times M \to(0, \infty)$,
and let $\kappa: \Om\to\Om'$ be a sufficient statistic of the model.

Then $(M, \Om, \mu, \mathbf{p})$ admits a Fisher--Neyman factorization,
that is, it is of the form (\ref{Fisher-Neyman}) in Example~\ref
{ex:sufficient} for some measure $\mu_0$ on $\Om$, $\mu_0' := \kappa
_*(\mu_0)$ and some function $\phi': \Om' \times M \to(0, \infty)$.
\end{proposition}

\begin{proof}
If $\mathbf{p}(\xi) = p(\cdot;\xi) \mu$ where $p$ is positive and
differentiable in the $\xi$-variable, then $\log p(\cdot; \xi)$ and
$\log p'(\cdot; \xi)$ are well defined functions on $\Om\times M$
and $\Om' \times M$, respectively and differentiable in $\xi$. In
particular, $\kappa^*(\p_V \log p'(\cdot; \xi)) = \p_V(\log
p'(\kappa(\cdot);\xi))$, so that by Theorem~\ref
{thm:induced-kintegrable} equality in (\ref{eq:est-log}) holds for $k
> 1$ iff
\[
\p_V \log\frac{p(\cdot; \xi)}{p'(\kappa(\cdot);\xi)} = \p_V \bigl(\log p(\cdot;
\xi) - \bigl(\log p'\bigl(\kappa(\cdot);\xi\bigr)\bigr)\bigr) = 0.
\]
If $M$ is connected, then this is the case for \emph{all} $V \in TM$ iff
the positive function\vspace*{2pt} $h(\cdot) := \frac{p(\cdot; \xi)}{p'(\kappa
(\cdot);\xi)}$ does not depend on $\xi\in M$. Thus, setting $\mu_0
:= h \mu$ implies (\ref{Fisher-Neyman}), showing the assertion.
\end{proof}

Observe that the proof uses the positivity of the density function $p$
in a crucial way. In fact, without this assumption the conclusion is
false, as the following example shows.

%ex5.2 #&#
\begin{example}\label{ex:suff}
Let $\Om:= (-1,1) \times(0,1)$, $\Om' := (-1,1)$ and $\kappa: \Om
\to\Om'$ be the projection onto the first component. For $\xi\in\R
$ we define the statistical model $\mathbf{p}$ on $\Om$ as $\mathbf{p}(\xi)
:= p(s,t; \xi) \,  ds\,  dt$, where
\[
p(s,t; \xi) := %
\begin{cases} h(\xi) & \mbox{for }\xi\geq0\mbox{ and
}s \geq0,
\\
2 h(\xi) t & \mbox{for }\xi< 0\mbox{ and }s\geq0,
\\
1 - h(\xi) & \mbox{for }s < 0, \end{cases} %
\]
with $h(\xi) := \exp(-|\xi|^{-1})$ for $\xi\neq0$ and $h(0) := 0$.
Then $\mathbf{p}(\xi)$ is a probability measure, and
\[
\mathbf{p}'(\xi) := \kappa_*\mathbf{p}(\xi) = p'(s;
\xi)\, ds\qquad \mbox{with } p'(s; \xi) := \bigl(1-h(\xi)\bigr)
\chi_{(-1,0)}(s) + h(\xi) \chi_{[0,1)}(s),
\]
and thus,
\[
\bigl\|\p_\xi\log p(s,t; \xi)\bigr\|_k =
\bigl\|\p_\xi\log p'(s; \xi)\bigr\|
_k = k \biggl(\biggl\|\frac{d}{d\xi} h(\xi)^{1/k}
\biggr\|^k + \biggl\|\frac{d}{d\xi} \bigl(1-h(\xi)
\bigr)^{1/k}\biggr\|^k \biggr)^{1/k},
\]
where the norm is taken in $L^k(\Om, \mathbf{p}(\xi))$ and $L^k(\Om',
\mathbf{p}'(\xi))$, respectively. Since this expression is continuous in
$\xi$ for all $k$, the models $(\R, \Om, \mathbf{p})$ and $(\R, \Om',
\mathbf{p}')$ are $\infty$-integrable, and there is no information loss
of $k$th order for any $k \geq1$, so that $\kappa$ is a sufficient
statistic of the model in the sense of Definition~\ref{def:suffstat2}. Thus, $\kappa$ is a sufficient statistic for the model.

Indeed, this model admits a Fisher--Neyman factorization when restricted
to $\xi\geq0$ and to $\xi\leq0$, respectively; in these cases, we have
\[
\mathbf{p}(\xi) = p'(s; \xi) \mu_\pm,
\]
with the measures $\mu_+ := ds \, dt$ for $\xi\geq0$ and $\mu_- :=
(\chi_{(-1,0)}(s) + 2t \chi_{[0,1)}(s)) \,  ds\,  dt$ for $\xi\leq0$,
respectively.

However, since $\mu_+ \neq\mu_-$, $\kappa$ is \emph{not} of the form
(\ref{Fisher-Neyman}) and hence \emph{not among the sufficient
statistics given in Example}~\ref{ex:sufficient} when defining it for
all $\xi\in\R$. This does not contradict Proposition~\ref
{prop:infoloss} since $p(s,t; \xi)$ is not positive a.e. for $\xi= 0$.
\end{example}

The reader might be aware that most texts use the description (\ref
{Fisher-Neyman}) in Example~\ref{ex:sufficient} as a definition for a
sufficient statistic, for example, \cite{Fisher1922}, \cite{AJLS}, Definition~3.1, \cite{AN2000}, (2.17), \cite{Borovkov1998}, Theorem~1, p.~117. In the light of the Fisher--Neyman factorization in
Proposition~\ref{prop:infoloss}, this is equivalent to our Definition~\ref{def:suffstat2} under the assumption that the model is given by a
regular positive density function, an assumption that has been made in
all these references.

However, the significance of Example~\ref{ex:suff} is that the two
notions of sufficiency are no longer equivalent if the assumption of
positivity of the density function is dropped. But since in this
example, all statistical information of $(M, \Om, \mathbf{p})$ can be
recovered from $(M, \Om', \mathbf{p}')$, it seems natural to define
sufficiency of a statistic in such a way that this example is subsumed,
that is, as in Definition~\ref{def:suffstat2}.

%\begin{appendix}
%\section{}
%\end{appendix}

% zodis "Acknowledgments" paliekamas pagal autoriu
\section*{Acknowledgements}
This work was mainly carried out at
the Max Planck Institute for Mathematics in the Sciences in Leipzig,
and we are grateful for the excellent working conditions provided at
that institution. H.V. L\^e is partially supported by Grant
RVO: 67985840. J. Jost acknowledges support from the ERC Advanced Grant
FP7-267087. We also thank the referees for numerous helpful comments
and suggestions which helped to significantly improve the manuscript.

\small{
%%%%%%%%%%%%%%%%%%%%%%%%%%%%%%%%%%%%

}

\begin{thebibliography}{999}
%%%%%%%%%%%%%%%%%%%%%%%%%%%%%%%%%%%%

\bibitem{Amari1980}{\sc  S. Amari}, Theory of information spaces. A geometrical foundation of statistics. POST RAAG Report 106, 1980. 
\bibitem{Amari1982}{\sc S. Amari}, Differential geometry of curved exponential families curvature and information loss. The Annals of Statistics, 10(1982),357-385.
\bibitem{Amari1987}{\sc S. Amari}, Differential Geometrical Theory of Statistics, in: Differential geometry in statistical inference, Institute of Mathematical Statistics, Lecture Note-Monograph Series, Volume 10,  California (1987).
\bibitem{AN2000}{\sc S. Amari, H. Nagaoka}, Methods of information geometry, Translations of mathematical monographs; v. 191, American Mathematical Society, Providence, RI; Oxford University Press, Oxford, 2000.
\bibitem{AJLS}{\sc N.Ay, J.Jost, H.V.L\^e, L.Schwachh\"ofer}, Information geometry and sufficient statistics, Probability Theory and Related Fields 162 no. 1-2, (2015), 327-364.
\bibitem{AJLS2}{\sc N.Ay, J.Jost, H.V.L\^e, L.Schwachh\"ofer}, Invariant geometric structures on statistical models, in: Geometric science of information, F. Nielsen, F. Barbaresco (eds.), LCNS 9398, Springer (2015)
\bibitem{Bauer}{\sc H. Bauer}, Measure and integration theory, translated from the German by Robert B. Burckel, de Gruyter, 2001
\bibitem{BBM}{\sc M. Bauer, M. Bruveris, P. Michor}, Uniqueness of the Fisher-Rao metric on the space of smooth densities, arXiv:1411.5577 (2014).
\bibitem{Borovkov1998} {\sc  A. A. Borovkov}, Mathematical statistics, Gordon and Breach Science Publishers, 1998.
\bibitem{CP2007}{\sc A. Cena and G. Pistone}, Exponential statistical model, AISM 59 (2007), 27-56.
\bibitem{Chentsov1965} {\sc N. Chentsov}, Category of  mathematical statistics, Dokl. Acad. Nauk USSR 164 (1965), 511-514. 
\bibitem{Chentsov1978}{\sc N. Chentsov}, Algebraic foundation of mathematical statistics,  Math. Operationsforsch. statist. Serie Statistics. v.9 (1978), 267-276.
\bibitem{Chentsov1982}{\sc N. Chentsov}, Statistical decision rules and optimal inference,  Moscow,  Nauka,  1972 (in Russian),   English translation  in:  Translation of Math. Monograph 53, AMS, Providence, 1982.
\bibitem{Efron1975}{\sc B. Efron}, Defining the curvature of a statistical problem (with applications to second order efficiency), with a discussion by C. R. Rao, Don A. Pierce, D. R. Cox, D. V. Lindley, Lucien LeCam, J. K. Ghosh, J. Pfanzagl, Niels Keiding, A. P. Dawid, Jim Reeds and with a reply by the author, Ann. Statist. 3 (1975),  1189-1242.
\bibitem{Fisher1922}{\sc R. A. Fisher}, On the mathematical foundations of theoretical statistics, Philosophical Transactions of the Royal Society of London. Series A 222(1922), 309-368.
\bibitem{GS1977} {\sc V. Guillemin, S. Sternberg}, Geometric asymptotics, Math. Surveys, 14, Amer. Math. Soc., Providence, R.I., 1977
\bibitem{Jeffreys1946} {\sc H. Jeffreys}, An invariant form for the prior probability in estimation problems, Proc. Roy. Soc. London. Ser. A. 186,  453-461, 1946.
\bibitem{Lang2002} {\sc S. Lang},  Introduction to Differentiable Manifolds, 2nd ed., Universitext, Springer (2002)
\bibitem{Lauritzen1987} {\sc S. Lauritzen}, Statistical manifolds, in: Differential geometry in statistical inference, Institute of Mathematical Statistics, Lecture Note-Monograph Series, Volume 10,  California (1987).
\bibitem{Le2013}{\sc H.V. L\^e}, The uniqueness of the Fisher metric as information metric, Ann. Inst. Statist. Math.  69, 879-896 (2017)
\bibitem{MS1966}{\sc N. Morse and  R. Sacksteder}, Statistical isomorphism,  Annals of Math. Statistics, 37 (1966), 203-214.
\bibitem{MC1991}{\sc  E. Morozova and N. Chentsov},  Natural geometry on families of  probability laws, Itogi Nauki i Techniki, Current problems of mathematics, Fundamental directions  83 (1991), Moscow, 133-265.
\bibitem{MR}{\sc M. Murray, J. Rice}, Differential geometry and statistics, Chapman \& Hall, London etc., 1993
\bibitem{Neveu1965}{\sc J. Neveu}, Mathematical Foundations of the Calculus of Probability, Holden-Day series in probability and statistics (1965)
\bibitem{Neyman1935} {\sc J. Neyman},Sur un teorema concernente le cosidette statistiche sufficienti, Giorn. Ist. Ital. Att., 6:320-334 (1935)
\bibitem{P2013} {\sc G. Pistone}, Nonparametric information geometry. Geometric science of information, Lecture Notes in Comput. Sci., 8085, Springer, Heidelberg (2013)
\bibitem{PS1995} {\sc G. Pistone  and C. Sempi}, An infinite-dimensional structure on the space of all the probability measures equivalent to a given one, The  Annals of Statistics 23 (1995), N. 5,  1543-1561. 
\bibitem{Rao1945}{\sc C. R. Rao}, Information and the accuracy attainable in the estimation of statistical parameters, Bulletin of the Calcutta Mathematical Society 37, 81-89, 1945.
\end{thebibliography}
\end{document}